\documentclass{siamart1116}

 \numberwithin{equation}{section}
  \numberwithin{theorem}{section}

 \usepackage{psfrag}
  \usepackage{soul}
  \usepackage[lofdepth,lotdepth]{subfig}

\usepackage{epsfig}
\usepackage{epstopdf}
\usepackage{graphicx} 
\usepackage{color}
\usepackage{amsmath}
\usepackage{amsfonts}
\usepackage{amssymb}%
\usepackage{latexsym}

  \def\nb{\nonumber}
\def \Vh0{\stackrel{\circ}{V}_h} 
   
\def\Om{\Omega}   
\newcommand{\q}{\quad}    \def\R{{\mathbb R}}
    
   \def\eps{\varepsilon}
\def\m{\mbox}

\newcommand{\lc}
{\mathrel{\raise2pt\hbox{${\mathop<\limits_{\raise1pt\hbox
{\mbox{$\sim$}}}}$}}}

\newcommand{\gc}
{\mathrel{\raise2pt\hbox{${\mathop>\limits_{\raise1pt\hbox{\mbox{$\sim$}}}}$}}}

\newcommand{\ec}
{\mathrel{\raise2pt\hbox{${\mathop=\limits_{\raise1pt\hbox{\mbox{$\sim$}}}}$}}}

\def\bb{\begin{equation}} \def\ee{\end{equation}}
\def\beqn{\begin{eqnarray}}  \def\eqn{\end{eqnarray}}
\def\beqnx{\begin{eqnarray*}} \def\eqnx{\end{eqnarray*}}

\newtheorem{assumption}[theorem]{Assumption}
\newtheorem{example}[theorem]{Example}
\newtheorem{remark}[theorem]{Remark}

\newtheorem{experiment}{Experiment}

\renewcommand{\theequation}{\thesection.\arabic{equation}}

\usepackage[a4paper]{geometry}
\newcommand{\hmaxell}{h_{\ell}}
\newcommand{\hbarkr}{\overline{h}(k,r)} 

\newcommand{\Ccont}{C_{\mathrm{cont}}}
\newcommand{\Ccoer}{C_{\mathrm{coer}}}
\newcommand{\Cpou}{1}

\newcommand{\tO}{{\Omega}}

\newcommand{\Vhl}{\mathcal{V}^h_\ell}
\newcommand{\tVhl}{{\mathcal{V}}^h_\ell}

\newcommand{\hprob}{h_{\mathrm{prob}}}
\newcommand{\epsprob}{\abs_{\mathrm{prob}}}
\newcommand{\Hprec}{H_{\mathrm{prec}}}
\newcommand{\epsprec}{\abs_{\mathrm{prec}}}

\newcommand{\cH}{{\cal H}}
\newcommand{\cI}{{\mathcal I}}

\newcommand{\cT}{{\mathcal T}}

\newcommand{\cD}{{\mathcal D}}
\newcommand{\cV}{{\mathcal V}}

\newcommand{\cN}{{\mathcal N}}

\newcommand{\cO}{{\mathcal O}}
\newcommand{\bx}{\boldsymbol{x}}

\newcommand{\bV}{\mathbf{V}}
\newcommand{\bU}{\mathbf{U}}
\newcommand{\bW}{\mathbf{W}}

\newcommand{\bn}{\mathbf{n}}

\newcommand{\bF}{\mathbf{F}}

\newcommand{\tOmega}{{\Omega}}

\newcommand{\ri}{{\rm i}}
\newcommand{\rd}{{\rm d}}

\newcommand{\beq}{\begin{equation}}
\newcommand{\eeq}{\end{equation}}
\newcommand{\beqs}{\begin{equation*}}
\newcommand{\eeqs}{\end{equation*}}
\newcommand{\bit}{\begin{itemize}}
\newcommand{\eit}{\end{itemize}}
\newcommand{\ben}{\begin{enumerate}}
\newcommand{\een}{\end{enumerate}}
\newcommand{\bal}{\begin{align}}
\newcommand{\eal}{\end{align}}
\newcommand{\bals}{\begin{align*}}
\newcommand{\eals}{\end{align*}}
\newcommand{\bse}{\begin{subequations}}
\newcommand{\ese}{\end{subequations}}
\newcommand{\bpr}{\begin{proposition}}
\newcommand{\epr}{\end{proposition}}
\newcommand{\bre}{\begin{remark}}
\newcommand{\ere}{\end{remark}}
\newcommand{\bpf}{\begin{proof}}
\newcommand{\epf}{\end{proof}}
\newcommand{\ble}{\begin{lemma}}
\newcommand{\ele}{\end{lemma}}
\newcommand{\bco}{\begin{corollary}}
\newcommand{\eco}{\end{corollary}}
\newcommand{\bex}{\begin{example}}
\newcommand{\eex}{\end{example}}
\newcommand{\bth}{\begin{theorem}}
\newcommand{\enth}{\end{theorem}}

\newcommand{\Rea}{\mathbb{R}}

\newcommand{\supp}{\mathop{{\rm supp}}}

\newcommand{\rI}{\mathrm{I}}

\newcommand{\gu}{\nabla u}

\newcommand{\gv}{\nabla v}

\newcommand{\LtG}{L^2(\GammaN)}

\newcommand{\tendi}{\rightarrow \infty}

\newcommand{\matrixA}{A}

\def\XXint#1#2#3{{\setbox0=\hbox{$#1{#2#3}{\int}$}
     \vcenter{\hbox{$#2#3$}}\kern-.5\wd0}}

\newcommand{\bv}{\overline{v}}

\newcommand{\bgv}{\overline{\gv}}

\newcommand{\GammaN}{\Gamma_I}

\newcommand*{\N}[1]{\left\|#1\right\|}

\newcommand{\oOmega}{\overline{\Omega}}

\newcommand{\matrixS}{{ S}}
\newcommand{\matrixM}{{ M}}

\newcommand{\matrixN}{{ N}}

\definecolor{darkred}{RGB}{139,0,0}
\definecolor{darkgreen}{RGB}{0,100,0}
\definecolor{darkmagenta}{RGB}{139,0,139}
\definecolor{darkpurple}{RGB}{110,0,180}
\definecolor{darkblue}{RGB}{40,0,200}

\usepackage{hyperref}
\definecolor{myblue}{rgb}{0,0,0.6}
\hypersetup{colorlinks=true,linkcolor=myblue,citecolor=myblue,filecolor=myblue,urlcolor=myblue}

\allowdisplaybreaks[4]

\newcommand{\tfa}{\text{ for all }}
\newcommand{\tfor}{\text{ for }}

\newcommand{\tand}{\text{ and }}
\newcommand{\noi}{\noindent}

\newcommand{\abs}{\eps}

\title{Domain  Decomposition with local impedance conditions for the Helmholtz equation with absorption
  \thanks{IGG  thanks the  
Department of  Mathematics at  
the Chinese University of Hong Kong for providing 
generous support and a stimulating research environment during 
his visits; he also thanks Paul Childs for first motivating him to study this problem. 
The authors thank Eero Vainikko (University of Tartu) for generously computing the  numerical experiments 
given in the paper.  We also thank  Eric Chung (Chinese University of Hong Kong) and Shihua Gong (University of Bath) for very  useful discussions. 
IGG acknowledges support from EPSRC grant EP/S003975/1,   EAS acknowledges support from EPSRC grant EP/R005591/1. and JZ acknowledges support from Hong Kong RGC General Research Fund (Project 14306718)
and NSFC/Hong Kong RGC Joint Research Scheme 2016/17 (N\_CUHK437/16).
We thank the anonymous referees for their   careful reading of the manuscript and perceptive comments  
which have improved the paper. }
}

\author{
I.~G.~Graham, 
\thanks{Department of Mathematical Sciences, University of Bath, Bath BA2
7AY,  UK \ \ 
}
  \and
E.~A.~Spence, 
\thanks{Department of Mathematical Sciences, University of Bath, Bath BA2 7AY,  UK \ \ 
}
\and 
  J.~Zou
  \thanks{Department of Mathematics, Chinese University of Hong Kong, Shatin,  N.T., Hong Kong\ \   
}
}

\begin{document}

\maketitle

\numberwithin{equation}{section}
\numberwithin{theorem}{section}
 
\begin{abstract}  We  consider one-level additive Schwarz preconditioners for a family of 
Helmholtz  problems  with absorption and   increasing   wavenumber $k$.   These problems are  discretized using the Galerkin method with    
nodal conforming  finite elements of any (fixed) order on  meshes with diameter $h = h(k)$,    
chosen  to maintain accuracy as $k$ increases. The action of the preconditioner requires solution of   independent (parallel) 
subproblems   (with impedance boundary conditions) on overlapping subdomains 
of diameter $H$ and overlap $\delta\leq H$. The solutions of these subproblems are linked together using    prolongation/restriction   
 operators defined using a  partition of unity; this formulation was previously proposed in [J.H. Kimn and M. Sarkis, {\em Comp. Meth. Appl. Mech. Engrg.}~196, 1507-1514, 2007]. In numerical experiments 
 (with $\delta \sim H$) for a model interior impedance problem, we observe robust (i.e.~$k-$independent) GMRES convergence as $k$ increases,
 with  $H\sim k^{-\alpha}$ and $\alpha \in [0,0.4]$  as  $k$ increases. This provides a highly-parallel, $k-$robust one-level domain decomposition method. We provide supporting theory by studying the preconditioner  applied to a range 
of absorptive problems, $k^2\mapsto k^2+ \ri \abs$, with absorption parameter $\abs$. 
Working in the Helmholtz ``energy''  inner product,  
and using the  underlying theory of  Helmholtz boundary-value problems,
we prove a $k-$independent  upper bound on the norm of the preconditioned 
matrix, valid for all   $\vert \abs\vert \lesssim  k^2$. 
We also prove a strictly-positive lower bound   
on the distance of the field of values of the preconditioned matrix from the origin which holds
when  $\eps/k$ is constant or 
growing arbitrarily slowly with $k$.  These results imply robustness of the preconditioner for the corresponding absorptive
problem as $k$ increases (given an appropriate choice of $H$). Since it is known that the absorptive problem provides a good preconditioner for the pure
Helmholtz problem when $\eps \sim k$,   our results  provide some theoretical support for
the observed robustness of  the preconditioner for the pure Helmholtz problem.
Since the subdomains used in  our preconditioner shrink only slowly  (relative to the fine grid size) as $k$ increases,
cheaper approximate  (two- or multi-level) versions of the preconditioner analysed here are important in practice  
and are reviewed here.
\end{abstract} 



\section{Introduction} \label{sec:intro}
The efficient solution of the wave equation  
is of intense current  interest because of the equation's  many applications 
(in, e.g., computational medicine, underwater acoustics, earthquake modelling, and seismic imaging).  
This paper concerns efficient iterative methods  for computing conforming finite-element approximations of any fixed order  
of the  Helmholtz equation (i.e.~the wave equation in the frequency domain)  
in 2-d or 3-d.  We formulate  and analyse parallel preconditioners for use with GMRES  and provide theory indicating   
that our preconditioners  should remain 
effective as the wavenumber $k$  increases.  

As $k$ increases, there are several difficulties that make the Helmholtz problem hard,  both mathematically and numerically: (i)  the solution  becomes more oscillatory and,  in general,  meshes need to be 
increasingly  refined,  leading to huge linear systems with  dimension growing at least with  $\mathcal{O}(k^{d})$; (ii)     
the linear systems become more indefinite; (iii) many ``standard'' preconditioning techniques that are motivated by positive-definite problems become unusable in practice; (iv) there is relatively little rigorous  
theory for justifying effective preconditioning of such large and indefinite problems.
Regarding (i), we recall that \cite[Chapter 4]{Ih:98} shows that in the linear finite
  element method for  a 1D Helmholtz problem,  $h\sim k^{-3/2}$ is necessary to ensure a bounded relative error as $k$ increases; the extension of this result to higher dimensions is given in \cite{LaSpWu:19}.

Our analysis is   carried out for the model Helmholtz problem with absorption:   
\begin{equation}\label{eq:PDE} 
-\Delta u  - (k^2+ \ri \abs)u = f \ , \end{equation} 
on an open  bounded  
polygonal (for $d=2$) or Lipschitz polyhedral (for $d=3$) domain   
$\Omega\subset \Rea^d$, 
with mixed boundary conditions
\begin{equation}\label{eq:ImpBC} 
\frac{\partial u }{\partial n} - \ri \eta  u = g  \quad \text{on} \quad \Gamma_I,  \quad \text{and} \quad  u = 0 \quad \text{on} \quad \Gamma_D , 
\end{equation}
where the wavenumber   $k>0$,  and $\Gamma = \Gamma_I \cup \Gamma_D$ is the boundary of $\Omega$, partitioned into $\Gamma_I$ and $\Gamma_D$, where $\Gamma_I$ has positive surface  measure. In applications, $k = \omega/c$,  with  $\omega$  the angular frequency and $c$  the wave speed. 
Here we restrict to the case when $c$ is a positive constant.  We allow the {\em absorption parameter}  $\abs$ 
to be negative, zero or positive (with $\abs = 0$ corresponding to the ``pure Helmholtz'' case);  more details on $\abs$ and   $\eta$  are given   in \S \ref{sec:fem_and_ddm}. 

{In practical wave scattering problems, the PDE \eqref{eq:PDE} is commonly posed 
on the infinite domain exterior to a bounded scatterer, which is then truncated using an artificial boundary. 
The significance of the impedance boundary condition in \eqref{eq:ImpBC} is that (with $\eta=\sqrt{k^2 + \ri \abs}$) it is the simplest possible approximation to the Sommerfeld radiation condition. The problem \eqref{eq:PDE}, \eqref{eq:ImpBC} can therefore model acoustic scattering by 
a sound-soft scatterer. 
Also included in \eqref{eq:PDE}, \eqref{eq:ImpBC} is the {\em interior impedance problem}, where $\Gamma_D = \emptyset$, and $\Gamma_I$ is the boundary of  $\Omega$. We assume that if $\Gamma_D \not = \emptyset$ then the surface measure of $\Gamma_D$ is positive.

The standard variational formulation for \eqref{eq:PDE}, \eqref{eq:ImpBC}  is:
Given $f \in L^2(\Omega)$, $g\in \LtG$,  find $u \in H^1_{0,D}(\Omega)  : = \left\{v \in H^1(\Omega): v = 0 \ \text{on} \, \Gamma_D \right\}$,  such that 
\beq\label{eq:vp}
a_\abs(u,v) = F(v) \quad \text{ for all }\,  v \in H^1_{0,D}(\Omega),
\eeq
where 
\beq\label{eq:Helmholtzvf_intro}
a_\abs(u,v) := \int_\Omega \gu \cdot \bgv - (k^2 + \ri \abs) \int_\Omega u \bv - \ri \eta  \int_{\GammaN} u\bv \quad\tand\quad F(v) := \int_\Omega f\bv + \int_{\GammaN} g \bv;
\eeq
when $\abs = 0$ and $\eta = k$ we write $a$ instead of $a_\abs$.
We approximate  \eqref{eq:vp} using the Galerkin method in a  conforming finite-element space $\cV^h \subset H^1_{0,D}(\Omega)$  (consisting of continuous piecewise polynomials of arbitrary fixed order), 
on a shape-regular mesh $\cT^h$ with mesh diameter $h$ (assumed to resolve the interface $\overline{\Gamma_I}\cap \overline{\Gamma_D}$ when this is non-empty and  points on the interface  are treated as Dirichlet points). This  yields  the linear system 
 \beq\label{eq:discrete}
 \matrixA_\abs \bU  \ := \ (\matrixS - (k^2+ \ri \abs)  \matrixM - \ri \eta \matrixN) \bU \ =\  \bF , 
 \eeq   
where $\bU$ is the vector of nodal values of the finite-element approximation $u_h \approx u$,  
$\matrixS$ is the stiffness matrix for the negative Laplace operator, $\matrixM$ is the 
domain mass matrix and $\matrixN$ is the boundary mass 
matrix (corresponding, respectively, to each of the terms in $a_\eps(u,v)$ in \eqref{eq:Helmholtzvf_intro}, and described in more detail 
in \S\ref{subsec:fem}).  $A_\abs $ is  large, sparse, and indefinite.  
When $\abs = 0$ and $\eta = k$ we write $A$ instead of $A_\abs$.
In common with many other investigations in the literature, we consider the situation where, following Point (i) on the previous page, $h$ is chosen  as a function of $k$ to maintain accuracy as $k$ increases (see Remark \ref{rem:accuracy} below for more details).

One way to understand the essential difficulty in 
preconditioning $A$ (as $k$ increases)  is to recall  
that the 
fundamental solution
of the  operator  in 
\eqref{eq:PDE} with $\eps=0$ (in three dimensions)   
is $ G(x, y) = \exp(\ri k  r)/r$, where  $r = \vert x - y\vert$,  
with    $\vert \cdot \vert$ denoting  the Euclidean norm,
and so 
a good preconditioner for \eqref{eq:PDE} with $\eps=0$  should,  roughly speaking,  approximate the  integral 
operator with kernel  $G$.
When $k = 0$  this operator is ``data-sparse'', since     the $j$th derivative of $G$ decays 
with order  $\mathcal{O}(r^{-(j+1)})$,  when  $x$ and $y$ are well-separated. Thus,  a  source in a given region is only felt weakly far 
away, a   fact that  underlies many successful preconditioners for 
Laplace-like  problems (e.g. multigrid, domain decomposition, or $\cH$-matrices). However,  when $k$ is large,     
the $j$th derivative of $G$ decays with the  much slower rate    $\mathcal{O}(k^{j} r^{-1})$, 
and the application of Laplace-like preconditioning strategies becomes problematic.
While directional clustering methods (see, e.g., \cite{EnYi:07}, \cite{BoMe:17} and the references therein) have been developed for  homogeneous Helmholtz  problems,  formulated using  boundary integral equations,
domain-based methods such as those considered here remain of great importance, due to  their
applicability  to general problems with sources and heterogeneities.   


Introducing absorption, $\abs \not = 0$, has the effect of improving the decay of the Green's function.
While absorptive problems do appear in applications (and our results here cover  these), our deeper motivation 
for including  $\abs $  is that it has proved useful for both constructing and providing the theory for 
preconditioners for  the case $\abs = 0$.  In    \cite{GaGrSp:15} it was proved 
(subject to certain natural conditions on  $\Omega$,    $h$ and $\abs$), that  
there is a constant $K$, independent of $h$, $k$, and $\abs$, such that  
\begin{align} \label{eq:GGS} \Vert I - A_\abs^{-1} A \Vert_2 \ \leq\  K \frac{|\abs|}{k}\ .
\end{align} 
Thus the  left-hand side of \eqref{eq:GGS} can  then be   made small  by choosing  $\abs$ to be a small-enough 
multiple of $k$.   However    $A_\abs^{-1}$ is not  
a practical preconditioner for $A$, and we therefore replace  it by a  approximation 
$B_\abs^{-1} \approx A_\abs^{-1}$. 
Using the classical results about GMRES in \cite{EiElSc:83}, we  say  that  
$B_\abs^{-1}$ is a {\em good preconditioner}  for $A$ if both (i) the matrix $B^{-1}_\abs A$ has 
Euclidean norm   bounded above, and (ii) the field of values  (in the Euclidean norm) 
of $B^{-1}_\abs A$ is bounded away from the origin, with both bounds independent of  $k$ and $\abs$. 
If both (i) and (ii) are satisfied then, by \cite{EiElSc:83}, GMRES for $B_{\eps}^{-1}A$ will converge in a number of iterations independent of $k$ and $\abs$.  

In order to characterize good choices of    $B_\abs^{-1}$,  we can write   
\begin{align} \label{eq:pert}  
B_\abs^{-1}A \ = \ B_\eps^{-1} A_\abs \ -\ B_\abs^{-1} A_\abs (I - A_\abs^{-1} A).
\end{align} 
Then \eqref{eq:GGS} combined with   \eqref{eq:pert} 
suggests that $B_\abs^{-1}$ will be a good preconditioner for $A$ provided that
 \begin{align}\label{eq:c}  \
B_{\abs}^{-1}  \ \text{is a good preconditioner for} \  A_{\abs}   \text{ when } {|}\abs{|}=ck, \text{ with $c$ sufficiently small.}
\end{align}
(An argument making this statement  rigorous is given in Appendix \S\ref{sec:app1}).


\subsection{The novel  results of the paper}
\label{subsec:main} 

We  give a new and rigorous proof that  $B_\abs^{-1}$ is a good preconditioner for $A_\abs$ when $B_\abs^{-1}$  is a
  simple  additive Schwarz preconditoner,
  constructed by solving  independent  (local) Helmholtz impedance subproblems on overlapping subdomains of $\Omega$, 
  linked  by   prolongation/restriction operators defined via a partition of unity (see  \S \ref{subsec:prec}).

   Theorem \ref{thm:main_intro} gives
general estimates for the norm and the distance of the   field of values  from the origin of
 $B_\abs^{-1}A_\abs$,  under the general 
assumption that the local solvers are sufficiently good approximate inverses for the localised global problem (assumption \eqref{localstar}). The estimates are explicit in the wavenumber $k$, the fine mesh diameter $h$, the number of overlaps $\Lambda$, the subdomain diameter $H$, the overlap size  $\delta$ and the absorption parameter $\abs$. Corollaries \ref{cor:upper_bound} and \ref{cor:lower_bound} then provide more concrete estimates
under additional conditions on  $H$,  $\delta$ and  $\abs$.   In particular:
\begin{itemize}
  \item[(1)] Corollary
\ref{cor:upper_bound}  provides conditions under which 
the norm of $B_\abs^{-1}A_\abs$  is uniformly bounded from above for all $0 \leq \vert \eps\vert  \leq k^2$
\item[(2)] Corollary \ref{cor:lower_bound} provides conditions under which the field of values of $B_\abs^{-1}A_\abs$
  is uniformly bounded away from the origin.  These hold (for appropriate $H = H(k)$), 
  when  ${|}\abs{|} \sim k^{1+\beta}$, with $\beta$ arbitrarily close to $0$ or when  
  ${|}\abs{|} = Ck$, for some large enough  constant $C$.
\end{itemize}
 Although  both the latter requirement
   in (2) and the requirement in \eqref{eq:c} suggest we should take  $\eps$ proportional to $k$, the required constants $C,c$
    are not explicitly known and so a rigorous lower bound   on the field of values can not be deduced in the   pure Helmholtz case. Nevertheless  numerical experiments in   \S \ref{sec:Numerical} still suggest that $B_0^{-1} $ is  a good preconditioner for the pure Helmholtz problem, for certain choices of $H(k)$, decreasing as $k \rightarrow \infty$.

Important features of the  results of Theorem \ref{thm:main_intro} and Corollaries \ref{cor:upper_bound} and \ref{cor:lower_bound} are that (a) they hold for
bounded polygonal or Lipschitz polyhedral domains and  cover sound-soft scattering problems, truncated 
using  first order absorbing boundary conditions; 
(b) the theory allows  finite element methods of any fixed order on shape-regular meshes;
and  general  shape-regular subdomains; (c) the  proof constitutes
a substantial extension of classical Schwarz theory to the non-self-adjoint case; 
 (d) via a duality argument,  the theory covers both left- and right-preconditioning simultaneously.

  To  achieve the  goal of a highly-parallel and provably $\cO(n)$
  solver for the Helmholtz equation as $k$ increases, one would need:  
\bit
\item[(i)] a $k$-independent (i.e.~$\cO(1)$) number of iterations, 
\item[(ii)] the action of the preconditioner to be $\cO(n)$, and 
\item[(iii)] (roughly speaking) the preconditioner to be as parallel as possible. 
\eit
Since we propose here a one-level additive Schwarz method,  (iii) is achieved.
The main achievement of our paper   
is  fundamental theory obtaining conditions under which (i) is achieved for Schwarz methods  
even without global coarse solver.
 As can be seen from the experiments in \S \ref{sec:Numerical},
the  subdomain size $H$ needed to  ensure 
robustness can shrink to zero as $k$ increases, but remains large
relative to the fine grid size $h$.  Further work is needed to achieve requirement (ii).  However, in \S \ref{sub:pract} we briefly discuss
  the cost of the subdomain problems,  together with ways of reducing this cost (some of which have been recently implemented and tested  \cite{GrSpVa:17,GrSpVa:17a,BoDoGrSpTo:17a,BoDoGrSpTo:17b,BoDoGrSpTo:17c}).
  Relatively large subdomain problems are also encountered in sweeping-style preconditioners,
  although in this case, they are typically posed on ``quasi $(d-1)$ dimensional'' slices of an original $d-$dimensional domain,
 and efficient direct solvers have been developed for these (e.g. \cite{PoEnLiYi:13}).


Finally, we note that it is perhaps  remarkable that this  one-level additive Schwarz method
can be robust when the subdomain size $H \rightarrow 0$. 
This conflicts with standard intuition and existing understanding,    
even for self-adjoint coercive PDEs; there, if $H \rightarrow 0$,  
the condition number of the one-level preconditioned problem grows like $\cO((\delta H)^{-1})$.
In the Helmholtz case, however, we are solving a family of problems parametrized by $k$. 
Even though the problem itself becomes ``harder" as $k$ increases, the results of this paper show that the one-level preconditioner can still remain 
robust.

\subsection{The preconditioner}
\label{subsec:prec} 

Our algorithm is a variation of the simple  one-level additive 
Schwarz method and is   based on a set of open polyhedral    
subdomains   $\{\Omega_\ell\}_{\ell = 1}^N$,  
forming  an overlapping cover of $\Omega$.
We assume that each $\overline{\Omega_\ell}$  
is non-empty and is  a union of  elements of the mesh $\cT^h$. 
The key component of the  preconditioner for \eqref{eq:discrete} is the solution of discrete ``local''  
versions of \eqref{eq:PDE}: 
\begin{equation}\label{eq:PDElocal} 
-\Delta u  - (k^2+ \ri \abs)u = f  \quad \text{on} \quad \Omega_\ell ,   \end{equation} 
subject to  boundary conditions
\begin{align} \label{eq:impBClocal} \       \frac{\partial u }{\partial n} - \ri \eta  u = 0  \quad \text{on} \quad \partial \Omega_\ell 
\backslash \Gamma_D \ \ \ \text{(assumed non-empty)}, \quad \text{and} \quad  u = 0 \quad \text{on} \quad \partial \Omega_\ell 
\cap \Gamma_D  .
\end{align}
We assume that if $\partial \Omega_\ell 
  \cap \Gamma_D  \not = \emptyset$,  then it has positive surface measure. Because $\Omega_\ell$ consists
  of a union of fine grid elements,  $\partial \Omega_\ell 
  \cap \Gamma_D $ 
  then contains at least one fine grid element.

To connect these local problems, we use a partition of unity $\{\chi_\ell\}_{\ell = 1}^N$ with properties 
  \begin{align}
  \text{for each} \ \ell: \   \chi_\ell: \oOmega \rightarrow \R, \quad \supp \chi_\ell \subseteq {\oOmega_\ell} \quad  \text{and} \quad 
  0 \leq  \chi_\ell(\bx) \leq 1, \ \  \text{when } \bx \in \oOmega,   
  \label{POUstar}
\end{align}
and such that
$$ \sum_\ell  \chi_\ell (\bx)  \ = \ 1 \ \quad
\text{for all } \  \bx \in \overline{\Omega}. $$ (Here we define  $\supp \chi_\ell :=  \{ \bx \in \oOmega: \chi_\ell(\bx) \not =0\}$.) 

The finite-element space $\cV^h\subset H_{0,D}^1(\Omega)$ underlying \eqref{eq:discrete} is assumed to have a nodal basis so that   each  $v_h \in \cV^h$  
is uniquely determined by its values $\{V_p : = v_h (\bx_p), \ p \in \cI^h\}$,   at nodes   $\{\bx_p: p \in \cI^h\} \subset \overline{\Omega} $  (where $\cI^h$ is a suitable  index set).  Nodes on the subdomain  $\overline{\Omega_\ell}$ 
are denoted $ \{ \bx_p: p \in \cI^h(\overline{\Omega_\ell})\} $.    
Using this notation, we can define a restriction matrix  $R_\ell$ that uses $\chi_\ell$ to map  
a nodal vector defined on $\overline{\Omega}$ to a nodal vector on $\overline{\Omega_\ell}$:     
\beq 
\label{eq:explicit}  
(R_\ell \bV)_p \ = \ \chi_\ell(\bx_p) V_p , \quad p \in \cI^h(\overline{\Omega_\ell}) . 
\eeq
We denote by $A_{\abs,\ell}$ the matrix obtained by 
approximating \eqref{eq:PDElocal} and \eqref{eq:impBClocal} in  $\cV^h$ (restricted to $\overline{\Omega_\ell}$); this matrix is a local analogue of the matrix  $A_\abs$ in \eqref{eq:discrete}.
Our preconditioner for $A_\abs$ is then simply: 
\begin{equation} \label{eq:ASpc} 
B_\abs^{-1}  :=  \sum_{\ell = 1}^N 
R_\ell^\top (A_{\abs,\ell})^{-1} R_\ell  \ ,
\end{equation}
where  $R_\ell^\top $ is the   transpose of $R_\ell$.  
Hence the action of $B_\eps^{-1}$  consists of $N$ parallel ``local impedance solves''  added up with the aid of  
appropriate restrictions/prolongations.  
$B_\abs^{-1}$ 
coincides with the {\em ``OBDD-H"}  preconditioner proposed (without theory) by
Kimn and Sarkis in \cite{KiSa:07} 
(also called SORAS -- 
  see, e.g., \cite[\S7.7.2]{DoJoNa:15}.)

In the analysis we use
  the $k-$dependent inner product and norm:  
\begin{align}\label{eq:wip}
  \langle \bV ,\bW \rangle_{D_k}:=  \bW^*D_k \bV, \quad \Vert \bV \Vert_{D_k} =  \langle \bV ,\bV \rangle_{D_k}^{1/2} , \quad \text{where} \quad D_k = (S+k^2 M). 
\end{align} 
In fact, $D_k$ is the  stiffness matrix arising from approximating via the Galerkin method in  $\cV^h$ the Helmholtz energy norm
\begin{align} \label{eq:HelmE} 
\Vert v \Vert_{1,k} := (v,v)_{1,k}^{1/2} \ \quad \text{where} \quad (v,w)_{1,k} \ := \ (\nabla v, \nabla w)_{L^2(\Omega)} + k^2 (v,w)_{L^2(\Omega)}.  
\end{align}
When $\widetilde{\Omega}$ is any subdomain of $\Omega$ we write $(\cdot, \cdot)_{1,k,\widetilde{\Omega}}$ and $\Vert \cdot \Vert_{1,k,\widetilde{\Omega}}$ for the corresponding inner product and norm on $\widetilde{\Omega}$.

\subsection{Related literature}
\label{subsec:related}  
There have been  two important recent  ideas that have had a large effect on the field of iterative solvers for the Helmholtz equation.  
The first is the ``shifted Laplace''  preconditioner, arising from initial ideas by   in \cite{BaGoTu:83} and   
\cite{LaGi:02}, and then developed and advocated in \cite{ErVuOo:04, ErOoVu:06,  VaErVu:07}. 
Since the  fundamental solution of \eqref{eq:PDE}  enjoys ``Laplace-like'' decay when $\abs$ is large enough, the  ``shifted Laplace'' preconditioner 
 uses  a multigrid approximation  of the absorptive problem to  precondition the ``pure Helmholtz''  problem $\abs = 0$.

The second  concerns   a class of multiplicative domain decomposition methods that fall 
under the general heading of ``sweeping'', e.g.~\cite{EnYi:11b, EnYi:11c, EnYi:11a, St:13,ChXi:13, PoEnLiYi:13, ZeDe:16, GaZh:16}. 
Restricting to a simple context, suppose  \eqref{eq:PDE} is  discretized on a tensor product grid
on  the unit square   and the unkowns are ordered lexicographically, yielding a 
block tridiagonal system matrix, each block corresponding to a row of nodes.
Sweeping methods can be thought of as    
approximate block-elimination methods for this system.
The Schur complement that  arises in  the block-elimination at a given line corresponds to
the solution of a Helmholtz problem in the domain below that line,
and  these problems  can be suitably truncated, to thinner strips  
either by ``moving {perfectly-matched layer} (PML)'' or $\mathcal{H}-$matrix approximation.
  The polarized trace algorithm \cite{ZeDe:16} takes this idea a step further
by pre-computing and compressing the solution operators on each strip, expediting the online process.

  
Both these ideas  have led to computation of challenging industrial strength applications, 
but neither of them have a rigorous theory.
For  ``sweeping'', 
the underpinning physical principle   applies only to  rectangular 2-d domains and tensor-product  discretizations 
(since the relevant low-rank result \cite{MaRo:07}  does not hold
 for  general domains and discretizations \cite{EnZh:18}), 
 and to the elimination of nodes in blocks, each consisting  of a small number of rows. Although
 the overarching principle of sweeping methods is  serial, there have been considerable innovations to enhance parallel efficiency. For example  
\cite{PoEnLiYi:13, ZeDe:16} propose recursive subdivision of the inner solves in each multiplicative 
sweeping step. Very recently \cite{taus2019sweeps} proposed the ``L-sweeps''  algorithm in which information propagates in
a 90-degree cone, allowing checkerboard domain decomposition.  In  \cite{LeJu:19} an overlapping domain decomposition solver
is proposed, with independent  subdomain solves at each step.

On the other hand the ``shifted Laplace'' algorithm    
is not in general robust with respect to $k$, 
since the  choice ${|}\abs{|} \sim k^2 $,
which is needed to make multigrid work \cite{CoGa:17},
turns out to be too large a perturbation of the pure Helmholtz problem to remain robust as $k \rightarrow \infty$.
Although recent enhancements based on deflation     \cite{ShLaVu:13, ShLaRaNaVu:16, ErRaNa:17,DwVu:18}
have greatly improved  the shifted Laplace preconditioner, a full theory is still missing.
A recent survey  of  shifted Laplace and related preconditioners is given in \cite{LaTaVu:17}.

Domain decomposition methods offer the attractive feature that their coarse grid and local 
problems can be adapted to allow for ``wave-like'' behaviour. 
There is a large 
literature (mostly empirical) on this (see, e.g., \cite{BeDe:97,Fa:00,GaMaNa:02, KiSa:07,KiSa:11, HuPeSc:13, HuSc:14, GaZh:16a}). A recent example is \cite{LuXu:19},  which proposes a multiplicative overlapping domain decompositon method
  as  a smoother in a multigrid algorithm for Helmholtz problems discretized by the continuous interior penalty method.      However there is no rigorous theory when $k$ is large,  for methods with \emph{either} many subdomains of general shape \emph{or} coarse grids.
The paper \cite{GrSpVa:17} provided the first such rigorous analysis for the problem with absorption, but the bounds for ${|}\abs{|}\ll k^2$ in \cite{GrSpVa:17} were very pessimistic.
The current paper extends this line of research to the case when wave-like components are inserted into the domain decomposition method. The results we obtain for the one-level method (i.e.~with no coarse {solver}) with impedance boundary conditions on the subdomains give practical bounds for much lower levels of absorption than in \cite{GrSpVa:17}.

Finally, we remark that  domain decomposition methods (with and without global coarse solver)
  for the  case when  $k$ is fixed and $h \rightarrow 0$ are in principle analysed in  Cai and Widlund \cite{CaWi:92}, since for small enough $h$ the Laplacian becomes the
dominant term in the discrete Helmholtz equation. However the current paper concentrates instead
on analysis for the more challenging case that allows   $k \rightarrow \infty$. 

  \subsection{Cost of the Preconditioner}\label{sub:pract}
  Here we discuss the cost of the preconditioner \eqref{eq:ASpc},
along with its possible approximations. We also  give a brief
  comparison with other preconditioners.    
The action of \eqref{eq:ASpc}
  requires the solution of $\cO(H^{-d})$ subproblems each of size
  $\cO((H/h)^d)$, with $d$ being the physical dimension.
  If $h \sim k^{-\gamma}$ with $\gamma \geq 1$ and $H \sim k^{-\alpha}$, with $0<\alpha < \gamma$,
  then the dimension of the global system grows  quickly with $k$, having dimension $n \sim k^{\gamma d}$.
The action of the preconditioner then requires the solution of  $\cO(k^{\alpha d})$ subproblems, each of size
$\cO(k^{(\gamma - \alpha )d})$. Since  $k \sim n^{1/\gamma d}$, this is equivalent to
\begin{align} \label{comp1} 
\cO(n^{\alpha/\gamma})\quad  \text{independent subproblems, each of size } \quad \cO(n^{(1- \alpha/\gamma)}).  
\end{align}

In the case $\alpha = 0.5$ (seen in Table \ref{tab:4} in \S\ref{sec:Numerical} to have an iteration count growing slowly with   $k$)
and $\gamma = 1.5$ (needed for accuracy of linear elements), the preconditioner has
\begin{align} \label{comp2} 
\cO(n^{1/3})\quad  \text{independent subproblems, each of size } \quad \cO(n^{2/3}).  
\end{align}

In the case $\alpha = 0.5$ and $\gamma = 1$ (e.g. a fixed number of grid points per wavelength, commonly used in practice and reasonable for higher order methods), the preconditioner has
\begin{align} \label{comp3} 
\cO(n^{1/2}) \quad  \text{ independent subproblems, each of size } \quad \cO(n^{1/2}) . 
\end{align} 
These subproblem sizes are comparable to those arising  from the (very successful) sweeping preconditioners, although (as pointed out above) the systems arising in sweeping are on thin rectangular subdomains and hence have beneficial special structure. In sweeping methods,  an
approximate inverse of $A$ is computed by an approximate $LDL^\top$ factorization. In the moving PML variant
(formulated  for a cubic domain with tensor product grid and appropriate boundary conditions)
one solves sequentially $\cO(n^{1/d})$ subproblems (on slices of the domain),  each of dimension
 $\cO(n^{(1-1/d)})$. 
When $d = 3$ this  coincides with  \eqref{comp2} and when $d=2$
 it coincides with  \eqref{comp3}.
The sweeping method in its basic format \cite{EnYi:11c} is multiplicative, whereas our preconditioner is fundamentally additive. On the other
hand sweeping provides an approximate inverse of the Helmholtz operator, while our aim here is only  to provide
 a good preconditioner (a somewhat weaker requirement). As a result,  our method is applicable in much more general geometrical situations.

Several other practical implementations of the  preconditioner analysed here have been tested.
For example,  
\cite{BoDoGrSpTo:17a}  reduced the subproblem size and added a coarse grid solver to reduce iteration count.
Although,  now not completely robust as $k$ (and hence $n$) increases, a slow  growth of iteration count of about
$\cO(n^{0.1})$ for 3D Helmholtz problems of size up to $n = 10^7$ was observed.
A similar method was  used for 3D Maxwell systems in \cite{BoDoGrSpTo:17b,BoDoGrSpTo:17c},
where good parallel performance was reported on systems of size up to
1 billion. Here the fact that absorption is added into the preconditioner turns out to be advantageous in
practice, since  the absorptive coarse grid problem can  be quickly solved with an inner iterative method and does not dominate the overall  cost.

Another approach  to reduce the cost of the preconditioner is to observe that the local \linebreak impedance solves are
local copies of the original problem (but on smaller domains and hence with  smaller effective wavenumber). This
allows them to be quickly resolved by an (inner) preconditioned GMRES combined with the same preconditioner. A
preliminary (serial) implementation of this method is given in \cite[\S 5.2.2]{GrSpVa:17a} where, on a 2D domain of size $\cO(1)$ the outer preconditioner was formulated on subdomains of size $\cO(k^{-0.4})$ and the inner preconditioner on subdomains of size $\cO(k^{-0.8})$. This was implemented on a fine discretization with
10 grid points per wavelength, in which only very small  problems of size $\cO((k^{-0/8}/ k^{-1})^2) = \cO(k^{0.4})$
had to be  solved directly. Results for $k $ up to $300$ are given in  \cite[\S 5.2.2]{GrSpVa:17a}, showing very low
inner iteration counts and outer iteration counts growing slowly ($\sim \cO(n^{0.2})$). For $k = 300$,  direct solvers were needed
for systems of size only  few hundred.  
The idea of recursive subdivision of subdomains also features heavily in efficient versions of sweeping
\cite{LiYi:16a}, and also in the polarized trace algorithm \cite{ZeDe:16}.

\subsection{Structure of the paper} 
In \S\S \ref{sec:2.1}, \ref{subsec:fem} we provide key   
estimates for the  local impedance solution operator at the continuous (PDE) level, 
and its  discretization. The properties of the preconditioner are established via its
interpretation as a sum of projections; this is set up  in \S \ref{sec:POU}.   
We prove the main results in  \S \ref{sec:Theory}  and present  
 numerical experiments in \S \ref{sec:Numerical}. In Appendix \ref{sec:app1} we give a rigorous basis for the discussion around \eqref{eq:c}. 

\section{Preliminaries}
\label{sec:fem_and_ddm} 

Throughout  we write $a\lesssim b$ when there exists a $C>0$, independent of all parameters of interest 
(here $\abs, k, h,H, \delta$, $\Lambda$, and  $\ell$ - with some of these defined later),  such that $a\leq Cb$. We write $a\sim b$ if $a\lesssim b$ and $b\lesssim a$. We make the following basic assumptions on $ k, \abs$ and $\eta$ throughout the paper. 
\begin{assumption} \label{ass:basic} 
The parameters $k$, $\abs$ and $\eta$ satisfy 
\begin{align} \label{eq:abs_limit} 
k \gtrsim 1, \quad 0 \leq \vert \abs\vert \leq   k^2 , \quad \text{and} \quad \vert \eta \vert  \sim k . 
\end{align} 
\end{assumption}

We recall the inequalities (valid for all $a,b > 0 $ and $\epsilon>0$), 
 \beq\label{eq:Cauchy}
 2ab\leq \frac{a^2}{\epsilon} + \epsilon b^2, \quad \text{and} \quad 
 \frac{1}{\sqrt{2}}(a+b) \leq \sqrt{a^2+b^2} \leq a+b . 
 \eeq

\subsection{A priori estimates}\label{sec:2.1}

The basic well-posedness of \eqref{eq:vp} is classical:

\begin{proposition}\label{prop:eu}
If either (i)\  $\abs>0$ and  $\Re(\eta)>0$, or (ii) \  $\abs<0$ and  $\Re(\eta)<0$,  
or (iii) \ $\abs=0$,   $\Re(\eta)\neq 0$,
the problem   
 \eqref{eq:Helmholtzvf_intro} has a unique solution. 
\end{proposition}
\bpf[Sketch proof]
For cases (i) and (ii),   uniqueness  can be   established  by {taking} $v = u$ {and $F=0$ in the weak form \eqref{eq:vp}} and {then} taking the imaginary part to show that $u = 0$.  Case (iii) is the standard ``pure Helmholtz'' case; uniqueness can be obtained by the unique continuation principle (e.g.~\cite[{Remark 8.1.1}]{Me:95}, {\cite[\S3]{GrSa:20}}). 
  Existence then follows for all cases via the Fredholm alternative,  since $a_\abs$ satisfies a G\v{a}rding inequality.
\epf 
 
In the domain decomposition method below  we will be  interested in local impedance solves on subdomains that may shrink 
in diameter as $k \rightarrow \infty$. For this reason we introduce the following.
  
\begin{definition}[Characteristic length scale]\label{def:cls}
 A domain  has \emph{characteristic length scale}  $L$ if its diameter $\sim L$, its surface area  $\sim L^{d-1}$, and its volume $\sim L^d$.
\end{definition}

\ble[Continuity and coercivity of the sesquilinear form $a_\abs$]\label{lem:cont_coer} 
\

\noi (i) Assume that $\Omega$ has characteristic length scale $L$  and that $\abs$ and $\eta$ satisfy \eqref{eq:abs_limit}. Then the sesquilinear form $a_\abs$ is continuous, i.e.
\beqs
\vert a_\abs(u,v)\vert \leq  \Ccont \Vert u \Vert_{1,k} \Vert v \Vert_{1,k}, \quad \text{ with }\quad
\Ccont \lesssim 
\left(1+ (kL)^{-1}\right),  \quad \text{for all} \ \  u,v \in H^1(\Omega).
\eeqs
\noi (ii) Let  $\sqrt{k^2 + \ri \abs}$ be defined via the square root with the branch cut on the positive real axis. If $\eta$ satisfies
\beq\label{eq:eta}
\Re \big( \eta\overline{\sqrt{k^2 + \ri \abs}}\big)\geq 0,
\eeq
 then  $a_\abs$ is coercive, i.e. 
\beqs
\vert a_\abs(v,v)\vert \ \gtrsim \ \Ccoer \Vert v \Vert_{1,k}^2, \quad\text{ with } \quad \Ccoer\sim \frac{|\abs|}{k^2},   \quad \text{for all} \ \  v \in H^1(\Omega).
\eeqs
\ele

\bpf
The assertion (ii) is Lemma 2.4 in \cite{GrSpVa:17} (note that the omitted constants in that result do not depend on $L$). The assertion (i) follows from the Cauchy-Schwarz inequality and  the multiplicative trace inequality, 
$
\N{v}^2_{L^2(\Gamma)} 
\lesssim \left( \frac{1}{L}\N{v}^2_{L^2(\Omega)} + \N{\gv}_{L^2(\Omega)}\N{v}_{L^2(\Omega)}\right)
$, 
(see, e.g., \cite[Last equation on p. 41]{Gr:85})
and the inequalities \eqref{eq:Cauchy}.
\epf

\begin{remark}[Adjoint coercivity]\label{rem:adj}
The definition of $\sqrt{k^2 + \ri \abs}$ implies that when $\eta$ is chosen to satisfy \eqref{eq:eta}, the coercivity constant for $a_\abs$ is exactly the same as the coercivity constant for the sesquilinear form for the adjoint problem obtained by 
replacing  $\abs$  by $-\abs$ and $\eta$  by  $ -\eta$.
\end{remark} 

\begin{definition}\label{def:25star}
    A Lipschitz open set $D$ is called \emph{starshaped with respect to a ball} if there exists 
    a point
  $\bx_0 \in D$ and a  $\gamma>0$ such that the position vector of any point $\bx \in D$ satisfies
$(\bx-\bx_0) \cdot \bn(\bx) \geq \gamma$  when  the normal vector $\bn(\bx)$ is defined;
see, e.g., \cite[Lemma 5.4.1]{Mo:11}.
\end{definition} 

\begin{theorem}[A priori bound on solution of \eqref{eq:vp}]
\label{thm:IIP_M}
Let $\Omega$ be starshaped with respect to a ball and have characteristic length scale $L$, and recall that 
we have assumed that $\Gamma_I$ has positive measure. Let $u$ be \emph{either} the solution to \eqref{eq:vp} with $f\in L^2(D)$ and $g=0$, \emph{or} the solution to the adjoint problem under the same assumptions on $f$ and $g$. Then, there exists $C_1,C_2$ (independent of $k,\abs,\eta,$ and $L$) such that
\beq\label{eq:3}
\N{u}_{1,k} \leq C_1L \N{f}_{L^2(\Omega)}, 
\quad \text{provided that} \quad 
\frac{|\abs| L}{k} \leq C_2.
\eeq
\end{theorem}
\bpf 
This result is essentially given by \cite[Theorem 2.9 and Remark 2.5]{GaGrSp:15}, except the dependence of the constants on $L$ is not kept track of there. To see that the condition 
${|}\abs{|}/k\leq c$ in \cite[Theorem 2.9]{GaGrSp:15} is really the right-hand inequality in \eqref{eq:3}, one needs to examine the  argument near  the end of the proof of \cite[Theorem 2.9]{GaGrSp:15}  (just before Remark 2.16) and observe that $R \,(:= \sup_{\bx\in \Omega}|\bx|) \sim L$. 
To see why the bound \eqref{eq:3} has the  factor of $L$ on the right-hand side, observe that choosing 
$\delta_3 = 1/(2R)$ and $\delta_4 \sim k^2$ in the proof of   \cite[Theorem 2.9]{GaGrSp:15} means that, 
in \cite[(2.29)]{GaGrSp:15},   the factor multiplying  $\N{f}^2_{L^2(\Omega)}$   is $\sim L^2$. (The $L$-explicit bound \eqref{eq:3} in the case $\abs=0$ is also obtained in \cite[Remark 3.6]{MoSp:14}.)
\epf

For simplicity, in the rest of the paper we assume that either $\eta={\rm sign}(\abs)k$ or $\eta= \sqrt{k^2+ \ri \abs}$; observe that both these choices satisfy
 the requirements on $\eta$ in \eqref{eq:abs_limit}, the conditions for uniqueness of the solution of \eqref{eq:vp} 
in Proposition \ref{prop:eu}, and the more-restrictive condition for coercivity \eqref{eq:eta} (see \cite[Remark 2.5]{GrSpVa:17}).

\subsection{Finite element method and subproblems} 
\label{subsec:fem} 
Let 
 $\cT^h$ be  a family of
conforming simplicial   meshes 
that are shape regular as the  
mesh diameter  $h \rightarrow 0$.  
A typical element of $\cT^h$ is written $\tau \in \cT^h$ and is considered as a closed subset of
$\overline{\Omega}$. 
Our approximation space $\cV^h$ is then the space of all continuous 
functions on $\Omega$ that  are polynomial of  
(total) degree $r-1$ with $r \geq 2$ (when restricted to any $\tau$) and  vanish on $\Gamma_D$. 
We assume we have a nodal basis for this space (for example the standard Lagrange basis), i.e.~with nodes  
$\cN^h = \{\bx_q:q \in \cI^h\}$, where $\cI^h$ is  a suitable index
set and corresponding basis  
$\{ \phi_p : p \in \cI^h \}$ with $\phi_p(\bx_q) = \delta_{p,q}$. 
For any continuous function $g$ on $\overline{\Omega}$,  we  introduce the standard 
nodal interpolation operator 
$\Pi^h g = \sum_{p \in \cI^h} g(\bx_p) \phi_p \ $. 
and we  assume 
the standard  error estimate (e.g. \cite[\S3.1]{Ci:78}):
\begin{align}
\Vert  (I - \Pi^h) v \Vert_{L^2(\tau)} + h \vert  (I - \Pi^h) v \vert_{H^1(\tau)} \ \leq \ C h^r \vert v \vert_{H^r(\tau)} ,  \quad \text{for all} \ v \in H^r(\Omega),    \label{eq:Ciarlet}
\end{align} 
for each $\tau \in \cT^h$, with $C$ indepedent of $\tau$,  provided $v \in H^r(\tau)$. 
The Galerkin approximation of \eqref{eq:vp} in the space 
$\cV^h$ is equivalent to the  linear system \eqref{eq:discrete}  
where  $F_\ell := \int_{\Omega} f \phi_\ell + \int_{\Gamma_I} g \phi_\ell$, and 
\begin{equation}\label{eq:matrices}
\matrixS_{\ell,m} = \int_{\Omega} \nabla \phi_\ell \cdot \nabla
\phi_m , \quad   \matrixM_{\ell,m} = \int_{\Omega}
\phi_\ell  \phi_m , \quad     
\matrixN_{\ell,m} = \int_{\Gamma} \phi_\ell  \phi_m, \quad  
\ell, m \in \cI^h\ . \end{equation}

We assume that the subdomains $\Omega_\ell$  introduced in \S \ref{subsec:prec}, 
 are 
Lipschitz polyhedra (polygons in 2-d) that are \emph{shape regular with parameter $H_\ell$} in the sense that 
each $\Omega_\ell$ has characteristic length scale $H_\ell$,  and we set $H = \max_\ell H_\ell$.
 In our analysis  we allow $H$ to depend on $k$ in such a way that $H$ could approach $0$ as $k\tendi$.
Some of the results below require  that each $\Omega_\ell$ is starshaped with respect to a ball, 
with the corresponding parameters $\gamma = \gamma_\ell$ in Definition \ref{def:25star} satisfying
$\gamma_\ell \geq \gamma_* > 0 $ for all $\ell$.
We 
describe this property by saying that the  $\Omega_\ell$ are 
\emph{starshaped  with respect to a ball, uniformly in $\ell$}.

Concerning the overlap, for each $\ell = 1, \ldots , N$, 
let $\mathring{\Omega}_\ell$ denote the
part of $\Omega_\ell$ that is not overlapped by any other subdomains. (Note that  $\mathring{\Omega}_\ell = \emptyset$ is possible.) 
For $\mu>0$ let $\Omega_{\ell, \mu}$ denote the set of points in
$\Omega_\ell$, every element of which is  a distance no more than $\mu$ from the interior 
boundary $\partial \Omega_\ell\backslash \Gamma $ . Then we assume that there exist constants  $0<\delta_\ell \lesssim H$ and $0<b<1$ such that, for each $\ell = 1, \ldots , N$,  
  \begin{equation}
  \label{eq:unifoverlap}\Omega_{\ell, b \delta_\ell }\subset \Omega_\ell\backslash
  \mathring{\Omega}_\ell \subset \Omega_{\ell,\delta_\ell };
\end{equation}
The case when $\delta_\ell \geq c H_\ell$  for some constant $c$ independent of  $\ell$ is called \emph{generous} overlap.
In Figure \ref{fig:subdomains} we depict a typical subdomain, with its  parts which are  overlapped by its neighbours and its (possibly) non-overlapped part. 

\begin{figure}[h]
\psfrag{Oi}{$\Omega_{\ell}$}
\psfrag{Oi0}{$\mathring{\Omega}_{\ell}$}
\psfrag{Oid}{$\Omega_{\ell, c \delta_\ell}$}
\psfrag{di}{$\delta_\ell$}
\psfrag{cdi}{$c \delta_\ell$}
\centerline{\includegraphics[scale=0.5]{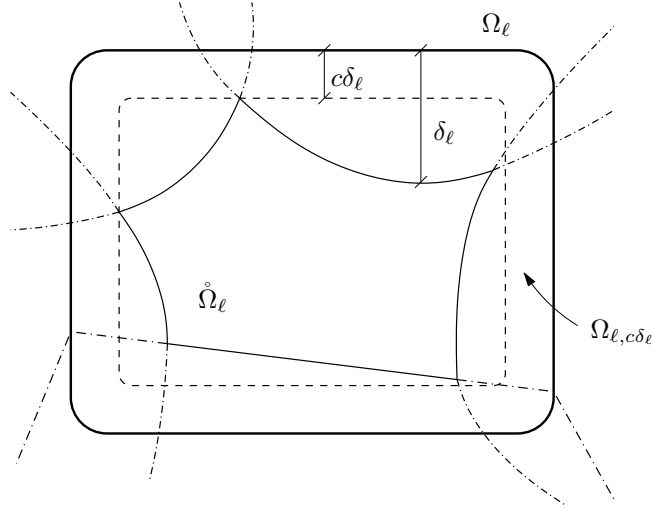}}
\caption{\label{fig:subdomains} The overlap parameter $\delta_\ell$, the 
``interior'' $\mathring{\Omega}_\ell$ and the ``near-boundary subset'' 
$\Omega_{\ell, c \delta_\ell}$ for a particular example of a 
subdomain $\Omega_\ell$, being overlapped by its neighbours.}
\end{figure}

We introduce the parameter
\begin{align} \label{eq:op} \delta : = \min_{\ell = 1, \ldots , N} \delta _\ell.  \end{align}

We  make the {\em finite-overlap assumption}: There exists a finite $\Lambda > 1$ independent of $N$ such that 
\begin{equation} 
\Lambda \ = \ \max \big\{ \#
  \Lambda(\ell): \ell = 1, \ldots , N\big\},    \quad \text{where} \quad   
\Lambda(\ell) = \big\{ \ell' :  \overline{\Omega_\ell} \cap \overline{\Omega_{\ell'}} \not =
\emptyset \big\} \ .   \label{eq:finoverlap}
\end{equation}       
It follows immediately from  \eqref{eq:finoverlap}  that, for all $v\in L^2(\Omega)$,
\beq\label{eq:finoverEuan1}
\sum_{\ell=1}^N \N{v}^2_{L^2(\Omega_\ell)} \leq \Lambda \N{v}^2_{L^2(\Omega)} \ \text{ and } \
\sum_{\ell=1}^N \N{v}^2_{1,k,\Omega_\ell} \leq \Lambda \N{v}^2_{1,k},  \  \text{when} \ \  v\in H^1(\Omega).
\eeq

For each $\ell$,  
we  introduce 
the space of finite-element functions on   
$\overline{\Omega_\ell}$ given by 
$\Vhl := \{v_h\vert_{\overline{\Omega_\ell}}: v_h \in \cV^h \} \ $. 
Recalling  that functions in $\cV^h$ vanish on the (outer)  Dirichlet boundary $\Gamma_D$,    
functions in $\cV^h_\ell$ also vanish on $\partial \Omega_\ell\cap \Gamma_D$ (which contains at least one element if it is non-empty),  but are otherwise
unconstrained.
The   local impedance sesquilinear form on $\Omega_\ell$ is 
\begin{equation}
\label{eq:localimp} 
a_{\abs,\ell}(v,w) \ := \ \int_{\Omega_\ell} \Big(\nabla v\cdot  \nabla \overline{w} - (k^2 + \ri \abs)  v \overline{w} \Big) \ - \ \ri \eta  \int_{   \partial \tOmega_\ell  \backslash \Gamma_D }  v \overline{w} \ , 
\end{equation}
for $v, w \in H^1_D(\tOmega_\ell) := \{ z \in H^1(\Omega_\ell): z = 0 \
  \text{on} \ \partial \Omega_\ell\cap \Gamma_D\}$.  For general \ $F_\ell \in (H^1(\Omega_\ell))'$, the continuous local impedance problem is: find $u_\ell \in H^1_D(\Omega_\ell)$ such that 
\beq\label{eq:Euan1} 
a_{\abs,\ell} (u_\ell , v_\ell) \ =\ F_\ell (v_\ell), \quad \text{for all} \quad v_\ell \in H^1_D(\Omega_\ell); 
\eeq
this problem is well-posed by  Proposition \ref{prop:eu} and its
finite-element approximation  is: find $u_{h,\ell} \in \cV^h_\ell$  such that
\begin{align} \label{eq:felocal}
a_{\abs,\ell} (u_{h,\ell} , v_{h,\ell}) \ =\ F_\ell (v_{h,\ell}), \quad \text{for all} \quad v_{h,\ell}  \in \cV^h_\ell . 
\end{align} 
The system matrix arising from \eqref{eq:felocal} is
$\big(A_{\abs,\ell}\big)_{i,j} := a_{\abs,\ell}(\phi_j,\phi_i) \quad \tfor i,j \in \cI^h(\overline{\Omega_\ell}).$

\begin{theorem}[Bounds on the solutions of the local problems \eqref{eq:felocal}] \label{thm:estimate} 

(i) For all $|\abs| > 0$, and  for any mesh size $h$,  
\eqref{eq:felocal} has a 
unique solution $u_{h,\ell}$ which satisfies
\begin{align} \label{eq:FEest1}
\Vert u_{h,\ell} \Vert_{1,k, \tO_\ell} \ \lesssim \ 
\Theta(\abs,H_\ell,k)
 \,  \max_{v_h \in \cV_\ell^h}  
\left( \frac
{\vert F(v_h) \vert}{ \Vert v_{ h} \Vert_{1,k,\tO_\ell} }\ \right),
\end{align} 
with
\beq\label{eq:Theta1}
\Theta(\abs,H_\ell,k)= k^2/\vert \abs\vert.
\eeq

(ii) If each $\Omega_\ell$ is starshaped with respect to a ball uniformly in $\ell$, then
for all $|\abs|\geq 0$, there exists a mesh threshold function $\hbarkr$  
such that when  $h \leq \hbarkr$,   \eqref{eq:felocal} has a unique solution $u_{h,\ell}$ which satisfies     
\eqref{eq:FEest1} with 
\begin{align}\label{eq:Theta2}
\Theta(\abs,H_\ell,k) = \min\left\{(1+kH_\ell), {k^2}/{\vert \abs\vert}\right\} \ , 
\end{align}
where we adopt the convention that $\Theta(0,H,k) = 1+Hk$.
\end{theorem}    

\begin{proof} The result (i) is a consequence of  
Lemma \ref{lem:cont_coer} and the Lax-Milgram lemma. 
The result (ii) follows from the fact (used in the case of Helmholtz problems by the authors of 
\cite{MeSa:10, MeSa:11} and their associated work) that when a sesquilinear form satisfies a G\aa rding inequality and the solution of the variational problem is unique, a ``Schatz-type'' argument obtains quasi-optimality under conditions on the approximability of the adjoint problem, and then the G\aa rding inequality can be used to verify a discrete inf-sup condition.  Indeed, following the proof of \cite[Theorem 4.2]{MeSa:10} and using the bound \eqref{eq:3} and the fact that $\Omega_\ell$ has characteristic length scale $ H_\ell$, we find that,  when $\vert \abs \vert H_\ell/k \leq C_2$, 
\begin{align} \label{eq:infsup}  \inf_{0 \not = v_h \in \cV_\ell^h} \, \sup_{0 \not = w_h \in \cV_\ell^h} 
\frac{\vert a_{\abs,\ell}(v_h, w_h)\vert}{\Vert v_h \Vert_{1,k} \Vert w_h \Vert_{1,k}}\ \ \geq \ 
\frac{1}{2 + \Ccont^{-1} + C_1 kH_\ell}.
\end{align}    
Then, from \eqref{eq:felocal},
 \beq\label{eq:Euan3}
 {\Vert u_{h,\ell} \Vert_{1,k,\Omega_\ell}} \ \lesssim\  {(1+ k H_\ell)}   \, \sup_{0 \not = v_h \in \cV_\ell^h} 
\frac{\vert F(v_h)\vert}{ \Vert v_h \Vert_{1,k,\Omega_\ell}},  
\eeq   
when $\vert \abs\vert  H_\ell /k \leq C_2$.
If  $\vert \abs \vert  H_\ell/k >  C_2$,  then $1 + H_\ell k > C_2 k^2/{|}\abs{|}$
and 
\eqref{eq:Theta2} follows from \eqref{eq:FEest1}.
\end{proof}

\begin{remark}[The mesh-threshold function $\hbarkr$]\label{rem:accuracy} 
Bounds on $\hbarkr$ are discussed  in detail in \cite[\S\S 5.1.2 and 5.2]{MeSa:11}.  For 2-d  
polygonal domains,   $k(hk/(r-1))^{r-1}$ is required to be sufficiently   small (see \cite[Equation 5.13]{MeSa:11}), equivalently $h$ being a sufficiently small multiple of $(r-1)k^{-(r/(r-1))} $. Therefore, when $r=2$ we require   
 $hk^2$ small, but the requirement relaxes as $r$ increases.   
 In 1-d, numerical experiments indicate that the requirement $hk^2$ sufficiently small  is necessary for quasioptimality 
 \cite[Figures 7-9]{IhBa:95}, \cite[\S4.5.4 and Figure 4.12]{Ih:98}.  The theoretical benefit of requiring $h \leq \hbarkr$ is that the estimate \eqref{eq:FEest1} holds uniformly over all choices of overlapping star-shaped  subdomains
   $\Omega_\ell$, each of which   has characteristic length $H_\ell$. However, to our knowledge, the requirement $h\sim k^{-2}$ is never imposed in practical  computations.       

  If one is only concerned with ensuring solvability,  a weaker requirement on $h$ arises.
   In 1-d, the relative error in both the $H^1$-semi-norm and the $L^2$-norm is bounded independently of $k$ if $hk^{3/2}$ is sufficiently small \cite[Equation 3.25]{IhBa:95}, \cite[Equation 4.5.15]{Ih:98}, with numerical experiments indicating that this is sharp \cite[Figure 11]{IhBa:95}, \cite[Figure 4.13]{Ih:98}.
Numerical experiments in \cite[\S3]{BaGoTu:85} showed that, at least for certain 2-d problems, the relative error in the $L^2$-norm is bounded independently of $k$
if  $hk^{3/2}$ is sufficiently small; this fact has recently been proved in \cite{LaSpWu:19}.
In \cite{DuWu:15},  under certain regularity assumptions,  it has been proven (in 2d and 3d) that,  if  $h^{2(r-1)} k^{2r-1} $  is small enough,  then
  the $H^1$ error is of 
  order $h^{2(r-1)} k^{2r-1} $.  Thus,  e.g., when $r = 2$, taking   $h \lesssim k^{-3/2}$  ensures
  that the
  problem is solvable and the error remains bounded as $ k$ increases.  This discussion is all for domains of diameter $\cO(1)$; for subdomains of decreasing diameter $\cO(H_\ell)$, the effective wavenumber is reduced to  $\cO(kH_\ell)$, and so the requirement on $h$ is even weaker.  
\end{remark} 

\subsection{Projection operators} \label{sec:POU} We now give more detail about the
    partition of unity $\{\chi_\ell\}$ and the restiction and prolongation matrices $R_\ell, R_\ell^\top$
discussed  in \S \ref{subsec:prec}. 
Note that since the subdomains are assumed to be unions of fine grid elements, their boundaries (and the boundaries of their supports) are fine-grid dependent. This is standard for domain decomposition methods (e.g. 
\cite[p. 57]{ToWi:05}).   
We choose the functions $\chi_\ell$ to be continuous piecewise linear on the mesh $\cT^h$, satisfying     
\begin{equation}  
\label{eq:derivpou} \Vert  \nabla \chi_\ell  \ \Vert_{\infty, \tau} 
\ \lesssim \ \delta_\ell ^{-1}, \quad \text{for all } \quad \tau \in \cT_h,   \ , 
\end{equation}
where the hidden constant is also required to be   independent   of the element  $\tau$. A partition of unity
satisfying this condition  is explicitly constructed in \cite[\S3.2]{ToWi:05}. 

We will use the operator $\Pi^h\circ \chi_\ell$.  
In fact, if $w_{h,\ell} \in \cV^h_\ell$ with nodal values $\bW$, then 
\beqs
\Pi^h\big(\chi_\ell w_{h,\ell}\big) = \sum_{p \in \cI^h} \big( R^T_\ell \bW\big)_p \phi_p,
\eeqs
where $R_\ell$ is defined by \eqref{eq:explicit}, 
and thus $\Pi^h\circ \chi_\ell$ defines a prolongation from  $\cV_\ell^h$ to $\cV^h$. 

To  analyse the
preconditioner    \eqref{eq:ASpc},  
we 
define the projections 
$Q_{ \abs,  \ell}^h  : H^1(\Omega) \rightarrow \tVhl$, by requiring that, given $v\in H^1(\Omega)$,
$Q_{\abs,  \ell}^h  v \in \tVhl$ 
satisfies the equation 
\begin{equation} 
a_{\abs, \ell}(Q_{\abs,\ell}^h v , w_{h,\ell}) \ = \ a_{\abs}(v, \Pi^h(\chi_\ell w_{h,\ell}))
\quad \text{for all} \quad w_{h,\ell}  \in \tVhl.
\label{eq:impproj} 
\end{equation}
For $|\abs|>0$, $Q_{\abs,\ell}$ is well-defined by Part (i) of Theorem \ref{thm:estimate}. For $\abs=0$, $Q_{\abs,\ell}$ is well-defined for all $h\leq \hbarkr$ by 
Part (ii) of Theorem \ref{thm:estimate}. 
To combine the actions of these  local projections additively, we define the global projection by
\begin{align} 
 Q_{\abs}^h \  :=\  \sum _{\ell=1}^N \Pi^h (\chi_\ell Q_{\abs,\ell}^h)  , \label{eq:global} 
\end{align} 
where again, each term in the sum can be interpreted as an element of $H^1(\Omega)$.  
The following theorem  shows that  the matrix representation of  
$Q^h_{\abs}$ restricted to  $\cV^h$ coincides with the preconditioned matrix  $B_\abs^{-1}A_\abs$.
This result uses the weighted inner product defined in \eqref{eq:wip}. 

\begin{theorem}[From projection operators to matrices] \label{thm:matrixprec}
Let $v_h \in \cV^h$,  with nodal values given in the vector $\bV$. 
Then, for any $\ell$, when the function $Q_{\abs,\ell}^h v_h \in \cV^h_\ell$ is well-defined it has nodal vector  
\begin{align}\label{eq:nodal}\bW = A_{\abs,\ell}^{-1} R_\ell A_\abs  \bV\ .\end{align} 
Consequently, for any $u_h, v_h \in \cV^h$, 
\begin{align}\label{eq:nodal_add}
( u_h , Q_{\abs}^h v_h)_{1,k} \ = \  \langle \bU, B_\abs^{-1} A_{\abs} \bV\rangle_{D_k}.
\end{align}

\end{theorem}
\begin{proof} 
With $\bW$ as given in \eqref{eq:nodal},   we have    
$ (A_{\abs,\ell} \bW)_q  = (R_\ell A \bV)_q$, for all $q \in \cI^h(\overline{\Omega_\ell})$, and so (recalling the definition of $R_\ell$ in \eqref{eq:explicit}),  
$$\sum_{p \in \cI^h(\overline{\Omega_\ell})} a_{\abs,\ell} (\phi_p, \phi_q) W_p = \chi_\ell(x_q) \sum_{p \in \cI^h(\overline{\Omega})} a_\abs(\phi_p, \phi_q) V_p \ , \quad \text{for each} \quad q \in \cI^h(\overline{\Omega_\ell}).
$$
 Then, letting $w_h\in \cV^h_\ell, v_h \in \cV^h$ be defined by the nodal values $\bW$, $\bV$, we have 
$$ a_{\abs,\ell} (w_h, \phi_q) =   a_{\abs} (v_h, \chi_\ell(\bx_q) \phi_q) \ \quad \text{for each} \quad q \in \cI^h(\overline{\Omega_\ell}).$$ 
By multiplying by $v_h(\bx_q)$ and using the definition of $\Pi^h$ and summing over $q$, we then have that
$$ a_{\abs,\ell} (w_h, v_h ) =   a_{\abs} (v_h, \Pi^h(\chi_\ell v_h)) \ , \quad \text{for all} \quad v_h \in \cV^h\ .  
$$
The definition of $Q_{\abs,\ell}^h$ \eqref{eq:impproj} and uniqueness then imply that $w_h = Q_{\abs,\ell}^hv_h$ which proves \eqref{eq:nodal}. 
Recalling \eqref{eq:wip} and \eqref{eq:global},  we obtain as a consequence of \eqref{eq:nodal} that
\begin{align*}
  (u_h, Q_\abs^h v_h)_{1,k} =  \sum_\ell (u_h, \Pi^h(\chi_\ell Q_{\abs,\ell}^h v_h))_{1,k} 
   = \sum_\ell \langle\bU, R_\ell^\top A_{\abs,\ell}^{-1} R_\ell A_\abs \bV\rangle_{D_k} =  \langle\bU, B_\eps^{-1}  A_\abs \bV\rangle_{D_k}. 
\end{align*}

\vspace{-0.3cm} 

\end{proof}

\section{The  Main Results}
\label{sec:Theory}

\subsection{Estimates involving the  overlapping decomposition}\label{sec:3.1}
  
\begin{lemma}[Estimates on norms involving $\chi_\ell$]\label{lem:tech}
  With $\delta_\ell$ as defined in \eqref{eq:unifoverlap},  
\begin{align}
  \Vert \chi_\ell v \Vert^2_{1,k,\Omega_\ell} - 2 \Vert v \Vert_{1,k,\Omega_\ell}^2 &  \ \ \lesssim
               \ \frac{1}{(k\delta_{\ell})^2}   \Vert v \Vert^2_{1,k,\Omega_\ell}  ,  \quad
                                                                                      \tfa\,v \in H^1(\Omega_\ell).   \label{eq:tech1}\\
  \sum _{\ell=1}^N \Vert \chi_\ell v \Vert_{1 ,k,\Omega_\ell}^2  &  \ \lesssim \Lambda\left( 1 + \frac{1}{(k\delta)^2}\right)  \
\Vert v \Vert_{1,k}^2   
 \quad \tfa\,v \in H^1(\Omega).  \label{eq:tech2}\\
\sum _{\ell=1}^N \Vert \chi_\ell^2 v \Vert_{1 ,k,\Omega_\ell}^2  &  \ \lesssim\Lambda\left( 1 + \frac{\Cpou}{(k\delta)^2}\right)^2  \  \Vert v \Vert_{1,k}^2   
 \quad \tfa\,v \in H^1(\Omega) . \label{eq:tech2a}\\
\sum_{\ell=1}^N \Vert\chi_\ell f \Vert^2_{L^2(\Omega_\ell)} & \ \geq \  \frac{1}{\Lambda}\N{f}^2_{L^2(\Omega)} \quad\tfa f \in L^2(\Omega). 
\label{eq:tech3}\\
  \sum_{\ell=1}^N \Vert \chi_\ell f \Vert^2_{1, k, \Omega_\ell} & \ \geq \   \frac{1}{\Lambda}\N{f}^2_{1,k} - 
C \frac{\Lambda }{k\delta}\N{f}_{1,k}^2 \quad \tfa f \in H^1(\Omega),
\label{eq:tech4}
\end{align}
where $C$ denotes a parameter-independent constant.
\end{lemma} 

\begin{proof} 
Using 
$\nabla (\chi_\ell v) = (\nabla \chi_\ell )v+\chi_\ell \nabla v$,
\eqref{POUstar}
and  \eqref{eq:derivpou}, we have that, for some constant $C$,  
$$\vert \nabla (\chi_\ell v)(\bx)  \vert^2 \ \leq\  2 \left(\frac{C}{\delta_\ell^2} \vert v(\bx) \vert^2 + \vert \nabla v (\bx) \vert^2\right),$$
 for all $\bx \in \Omega_\ell$. Then 
$$  \Vert \chi_\ell v \Vert_{1,k,\Omega_\ell}^2 \ 
\leq \ \frac{2C}{\delta_\ell^2}  \Vert v \Vert_{L^2(\Omega_\ell)}^2 +  2 \vert v \vert_{H^1(\Omega_\ell)}^2 + k^2 \Vert v \Vert_{L^2(\Omega_\ell)}^2 \ \leq 
 2 \left( 1 + \frac{C}{(k\delta_\ell)^2}\right)  
\Vert v \Vert_{1,k,\Omega_\ell}^2\,, $$
which yields  the estimate \eqref{eq:tech1}.   

From \eqref{eq:finoverEuan1} and \eqref{eq:op}, we see that \eqref{eq:tech2} follows from \eqref{eq:tech1}.
The estimate \eqref{eq:tech2a} follows from two successive applications of \eqref{eq:tech1}, summing both sides of the resulting estimate over $\ell$, and then using \eqref{eq:finoverEuan1} and \eqref{eq:finoverEuan1}.

To prove \eqref{eq:tech3}, first 
define, for each $\bx \in \oOmega$, a positive integer $m = m(\bx)$ by 
\beq\label{eq:defm}
m(\bx) := \# \, \big\{ \ell \in \{1, \ldots, N\}: \bx \in \mathrm{supp} \, \chi_\ell \big\} \ .
\eeq
Note that,  because $\supp(\chi_\ell) \subseteq \oOmega_\ell$,  the assumption \eqref{eq:finoverlap} ensures  
$ 1 \leq m(\bx) \leq \Lambda$, for all $\bx \in \Omega$. Then,  for any integer 
$j \in \{ 1, \ldots , \Lambda\} $, we define the subset of $\Omega$:  
$ D_j :=  \{ \bx \in \oOmega : m(\bx) =j\},   $
so that $\bx \in D_j$ if and only if $\bx$ lies in the supports of exactly $j$ of the 
partition of unity functions $\{\chi_\ell\}$.    Corresponding to these we also define the index sets: 
\begin{align} \label{eq:index}  
\cD(j) \ = \ \big\{\ell \in \{ 1, \ldots , N\}: \mathrm{supp} \, \chi_\ell \cap D_j \not = \emptyset\big\} . 
\end{align} 
This notation is illustrated in \S \ref{sec:Numerical} in the context of the particular overlapping cover used there. As that example shows, some of the sets $D_j$ can  have zero Lebesgue measure as subsets of $\Omega$.

Then, we have
\begin{align}\label{eq:property1}
\oOmega=\bigcup_{j=1}^\Lambda D_j\, ,      \q \text{and} \q 
D_i\cap D_j.=\emptyset ~\m{if} ~ i\ne j\, , 
\end{align} 
Moreover, for all   
$j = 1, \ldots, \Lambda$,  
\beq\label{eq:property2}
 \sum_{\ell \in \cD(j)}  \chi_\ell  (\bx) = 1 \quad \text{when  } ~\bx \in D_j\, .
 \eeq
Then, noting that $\#\{ \ell \in \cD(j): \chi_\ell(\bx) \not = 0  \}  = j \leq \Lambda$ and using \eqref{eq:property2} and the Cauchy-Schwarz inequality  we obtain, for all $\bx \in D_j$, 
\begin{equation} 1 \ = \ \left(\sum_{\ell \in \cD(j)}  \chi_\ell  (\bx)\right)^2  \leq j\sum_{\ell \in \cD(j)}  \chi_\ell^2   (\bx)
\le \Lambda\sum_{\ell \in \cD(j)}  \chi_\ell^2   (\bx) \,.
\label{eq:sqest} 
\end{equation} 
Using \eqref{eq:property1},  \eqref{eq:sqest} and \eqref{eq:index},   we find
\begin{align*} 
& \sum_{\ell=1}^N \int_{\Omega_\ell} \chi^2_\ell(\bx) \vert f(\bx)\vert^2 \rd \bx\nb 
\ =\ \sum_{j=1}^\Lambda\sum_{\ell=1}^N   \int_{\Omega_\ell\cap D_j} \chi^2_\ell(\bx) \vert f(\bx)\vert^2 \rd \bx\nb 
\ =\ \sum_{j=1}^\Lambda\sum_{{\ell\in \cD(j)}}   \int_{ \Omega_\ell\cap D_j} \chi^2_\ell(\bx) \vert f(\bx)\vert^2 \rd \bx\nb  \\
&\mbox{\hspace{1cm}}= \  \sum_{j=1}^\Lambda   \int_{D_j} \bigg(\sum_{\ell\in \cD(j)} \chi^2_\ell(\bx)\bigg) \vert f(\bx)\vert^2 \rd \bx\nb   
 \ \ge \ 
\frac 1\Lambda \sum_{j=1}^\Lambda   \int_{D_j}  \vert f(\bx)\vert^2 \rd \bx\nb \ = \  \frac 1\Lambda \int_{\Omega} \vert f(\bx)\vert^2 \rd \bx,
\end{align*} 
 which is \eqref{eq:tech3}.  Finally, for  \eqref{eq:tech4}, we use \eqref{POUstar} and \eqref{eq:derivpou}  to obtain
\begin{align*} \Vert \chi_\ell f \Vert_{1,k,\Omega_\ell}^2  \ & = \ k^2 \Vert \chi_\ell f \Vert_{L^2(\Omega_\ell)}^2 + \Vert \chi_\ell 
\vert \nabla f \vert \Vert_{L^2(\Omega_\ell)}^2 + 2 \mathrm{Re}  \int_{\Omega_\ell} \chi_\ell f \nabla \chi_\ell . \nabla \overline{f}  
+ \Vert f \vert \nabla \chi_\ell\vert \Vert^2_{L^2(\Omega_\ell)} \\
& \geq 
k^2 \Vert \chi_\ell f \Vert_{L^2(\Omega_\ell)}^2 + \Vert \chi_\ell 
\vert \nabla f \vert \Vert_{L^2(\Omega_\ell)}^2 - \frac{C}{k \delta_\ell} \Vert f \Vert_{1,k,\Omega_\ell}^2,   
\end{align*}
and the result is obtained by summing, and using \eqref{eq:tech3},  \eqref{eq:finoverEuan1} and \eqref{eq:op}. 
\end{proof} 

\begin{remark}\label{rem:stable_cts}
The estimate \eqref{eq:tech2}  provides a   
``stable splitting'', i.e.~any  $v \in H^1(\Omega)$ has a decomposition {into components} $\chi_\ell v \in H^1(\Omega_\ell)$, {with}
$ v = \sum_\ell \chi_\ell v,  $
so that sum of  the squares of the  energies of the components is bounded in terms of the square of the energy of $v$, with a  constant that is independent of $k, h, H $ and $\delta$, provided only that  
$k \delta  \gtrsim  1$.
Corollary \ref{stable_splitting_fem} 
below  provides  an analogous stable splitting for  finite element functions.  
This result is perhaps a little surprising, since,  for  positive-definite elliptic 
problems,  families of subdomains with decreasing diameter do not enjoy this   
 property (and a coarse space is needed to restore it) \cite{ToWi:05}. 
Here the stable splitting holds without coarse space   
as $k \rightarrow \infty$ (i.e.~for a  family of Helmholtz problems of increasing difficulty).  
This includes for example,  subdomains of diameter   
$H \sim k^{-\alpha}$ with $\alpha \in [0,1]$ and overlap  $k^{-1}  \lesssim   \delta  \leq H$.   
\end{remark}

\begin{lemma}[Error in interpolation of $\chi_\ell w_h$]\label{lem:motivated} 
Given $\ell \in\{ 1, \ldots, N\}$, suppose   $v_h \in \cV^h_\ell$. Then      
\begin{align} 
\|( {\rm I} - \Pi^h) (\chi_l v_h)\|_{1, k, \Om_l}  \ \lc \ 
 \left(1 + kh_\ell \right)\,  \left(\frac{\hmaxell}{\delta_\ell}\right)   \, \Vert v_h\Vert _{H^1(\Om_l)} \ ,    
 \quad 
 \label{eq:interp_bnd}
\end{align}
where $\hmaxell \ :=  \ \max_{\tau \subset  \overline{\Omega_\ell}} h_\tau$, and 
the hidden constant is independent of $\ell$.
\end{lemma} 

\begin{proof} 
For each element $\tau\in {\cT}^h$ with $\tau \subset \overline{\Omega_\ell}$, from   \eqref{eq:Ciarlet} we have 
\begin{align}\label{eq:first} 
  \|({\rm I} -\Pi^h) (\chi_l v_h)\|_{L^2(\tau)}   + h_\tau \vert (I - \Pi^h) (\chi_l v_h) \vert_{H^1(\tau)} \
  \lc \ h_\tau^r |\chi_l v_h|_{H^r(\tau)} \  . 
\end{align}
Let   $\alpha $  be any  multi-index of order $\vert \alpha \vert = r$. 
  Since, on $\tau$, $\chi_\ell$ is of degree $1$ and  $v_h$ is of degree $r-1$,   
  the Leibnitz formula tells us that,  
  $D^{\alpha} (\chi_\ell v_h)$ consists of only a linear combination of  functions of the form  $
  ({D^\beta  \chi_\ell})( D^{\alpha - \beta} v_h), \  \text{for all multi-indices with} \  \vert \beta \vert = 1 $
  (with coefficients independent of $\tau$). Combining this with \eqref{eq:derivpou}, leads to 
  \begin{align}\label{eq:alphad}
    \vert \chi_\ell v_h \vert_{H^r(\tau)} \ \lesssim \ \delta_\ell^{-1}\vert v_h \vert_{H^{r-1}(\tau)}.   
    \end{align}
    Then, using    \eqref{eq:first} and an  element-wise inverse estimate for shape regular elements, 
\begin{align}
  k \Vert ({\rm I} - \Pi^h) (\chi_l v_h)\Vert_{L^2(\tau)}\ & \lc\   kh_\tau   \frac{h_\tau}{\delta_\ell} h_\tau^{r-2}     \vert v_h \vert_{H^{r-1}(\tau)}\  \lc \  kh_\tau \frac{h_\tau}{\delta_\ell}  \Vert v_h\Vert_{H^1(\tau)}. 
  \label{eq:pert2}
\end{align}
Similarly 
\begin{equation}\label{eq:pert3}
  |({\rm I} - \Pi^h) (\chi_l v_h)|_{H^1(\tau)}\   \lc \
     \frac{h_\tau}{\delta_\ell} h_\tau^{r-2}     \vert v_h \vert_{H^{r-1}(\tau)}\ \lesssim \  \frac{h_\tau}{\delta_\ell}  \Vert v_h\Vert_{H^1(\tau)}.
\end{equation}
Conbining \eqref{eq:pert2} and \eqref{eq:pert3} yields   
the result.   
\end{proof} 

We now specify a simplifying  assumption on $h, k,$ and $\delta$.
\begin{assumption}\label{ass:basic2} 
\begin{align} \label{eq:basic_ass}  kh \ \lesssim\  1\  \ \quad \text{and} \ \ k \delta \  \gtrsim \  1.   
\end{align}
\end{assumption}
The left-hand inequality in \eqref{eq:basic_ass}  simply says that the fine mesh
resolves the oscillatory solution (which is always needed for an accuracy  anyway - see Remark \ref{rem:accuracy}),
while the right-hand inequality requires that the overlap contains at least  one oscillation.
Clearly  \eqref{eq:basic_ass} is equivalent to: 
\begin{align} 
\label{eq:equiv} k \hmaxell  \ \lesssim\  1\  \ \quad \text{and} \ \ k \delta_\ell  \  \gtrsim \  1, \quad \text{for each} \ \ell, 
\end{align}
provided the hidden constants are independent of $\ell$, and this in turn implies that $h_\ell/\delta_\ell \lesssim 1$. This latter inequality requires that the overlapped part of any subdomain only needs to be large enough
with respect to the {\em local} fine mesh diameter $\hmaxell$.
We  retain  the ratio 
$\hmaxell/\delta_\ell$  in the error estimate  \eqref{eq:interp_bnd}, since in many situations this can approach $0$ as $k \rightarrow \infty$.

\begin{corollary} \label{stable_splitting_fem} 
Under Assumption \ref{ass:basic2}, for $v_h \in \cV^h$,
\begin{align*}  v_h \ &= \ \sum_{\ell = 1}^N \Pi^h (\chi_\ell v_h) \quad 
 \text{and } 
\quad \sum_{\ell =1}^N \Vert 
\Pi^h (\chi_\ell v_h) \Vert_{1,k,\Omega_\ell}^2 \  \lesssim \ 
                        \Lambda
                        \Vert v_h \Vert_{1,k,\Omega}^2\ .
\end{align*}
\end{corollary} 

\begin{proof} 
  Using the triangle inequality, and then  \eqref{eq:tech1},  \eqref{eq:interp_bnd},  and \eqref{eq:equiv}, 
  we have
  \begin{align} \Vert \Pi^h (\chi_\ell v_h) \Vert_{1,k,\Omega_\ell} \ & \leq  \Vert \chi_\ell v_h \Vert_{1,k,\Omega_\ell} + \Vert (I - \Pi^h)(\chi_\ell v_h) \Vert_{1,k,\Omega_\ell} \nonumber \\
    \label{eq:teq}
    & \lesssim \ \Vert v_h \Vert_{1,k,\Omega_\ell} + \Vert v_h \Vert_{H^1(\Omega_\ell)} \ 
                                                                                     \lesssim 
                                                                                     \Vert v_h \Vert_{1,k,\Omega_\ell}, 
\end{align}   


and the result follows by squaring, summing, and applying \eqref{eq:finoverEuan1}.
\end{proof}

The next result is a kind of converse to the stable splitting result  discussed in Remark \ref{rem:stable_cts}. 
\begin{lemma}\label{lem:norm_of_sum}
For each  $\ell = 1, \ldots, N$,  choose any  functions $v_{\ell} \in H^1(\Omega)$,  
with $\supp v_\ell \subset \overline{\Omega_\ell}$. Then    
\begin{equation}
\label{eq:norm_of_sum}
\quad \left\Vert \, \sum_{\ell=1}^N  v_{\ell} \, \right\Vert_{1,k}^2 \ \leq \Lambda 
\sum_{\ell=1}^N \Vert v_{\ell} \Vert_{1,k, \Omega_\ell}^2\ .
\end{equation}  
\end{lemma}
\begin{proof}
The proof follows almost verbatim  that of \cite[Lemma 4.2]{GrSpVa:17}, with  a little extra 
care needed  to obtain 
the explicit constant  $\Lambda$ on the right-hand side. 
\end{proof}

\subsection{Results about the projection operators}\label{sec:3.2}

In this subsection, we study  the projection operators $Q_{\abs,\ell}^h $ which were defined in \eqref{eq:impproj}. Our goal is a  bound on  the operator 
$Q_{\abs,\ell}^h  - \Pi^h \chi_\ell $ with respect to  
the Helmholtz energy norm $\Vert \cdot \Vert_{1,k}$ -- see Lemma \ref{cor:rate}. This bound is  a  key ingredient of our main results --  
Theorem \ref{thm:main1} (for  projection operators) and  Theorem \ref{thm:main_intro} (for  matrices).    

We first note that, when $w_{h,\ell} \in \cV^h_\ell$,  $\Pi^h(\chi_\ell w_{h,\ell})$ is supported on $\Omega_\ell$
  and vanishes on $\partial \Omega_\ell$. Thus, by \eqref{eq:impproj}, for all  $ w_{h,\ell}  \in \tVhl$ and $v \in H^1(\Omega)$,     
\begin{align*} 
a_{\abs, \ell}(Q_{\abs,\ell}^h v , w_{h,\ell}) \ = \ a_{\abs,\ell}(v, \Pi^h(\chi_\ell w_{h,\ell}))
\end{align*}
 and hence 
\begin{equation}
\label{eq:comm} 
a_{\abs, \ell}(Q_{\abs,\ell}^h v - \Pi^h(\chi_\ell v), w_{h,\ell}) \ =
 \ a_{\abs,\ell}(v, \Pi^h(\chi_\ell w_{h,\ell})) - a_{\abs,\ell}(\Pi^h(\chi_\ell v), w_{h,\ell}). 
\end{equation}
This shows that  $Q_{\abs,\ell}^h v - \Pi^h(\chi_\ell v)$ satisfies a 
local impedance  problem with  ``data'' given by the   
``commutator'' (appearing on the right-hand side of   \eqref{eq:comm}). To estimate this commutator we write 
\begin{align}  
 a_{\abs,\ell}(v, \Pi^h(\chi_\ell w_{h,\ell}))  - a_{\abs,\ell}(\Pi^h(\chi_\ell v), w_{h,\ell}) 
 & = \  a_{\abs,\ell}((I - \Pi^h)(\chi_\ell v), w_{h,\ell}) - a_{\abs,\ell}(v, (I - \Pi^h)(\chi_\ell w_{h,\ell})) \nonumber \\
 & \quad \quad + \ b_\ell(v, w_{h,\ell})\ ,  \   
\label{eq:comm_arrange} \end{align}    
\begin{align} 
\text{where} \quad \quad b_\ell(v,w) \ & := \ a_{\abs,\ell}(v, \chi_\ell w) - a_{\abs,\ell}(\chi_\ell v, w) \  
= \ (v, \chi_\ell w)_{1,k,\Omega_\ell}  - (\chi_\ell v, w)_{1,k, \Omega_\ell}  \nonumber \\
&= \  
\int_{\Omega_\ell} \nabla \chi_\ell . ( \overline{w} \nabla v - v \nabla \overline{w} )  \   .
\label{eq:defb} 
\end{align} 
The following  lemma  provides estimates for each of the terms on the right-hand side of   \eqref{eq:comm_arrange}.

\begin{lemma} \label{lem:estb} $\mbox{\hspace{1in}}$ \\
(i) For all $v,w \in H^1(\Omega_\ell)$, 
$$ {\vert  b_\ell (v,w) \vert }  \  \lesssim \ \ (k \delta_\ell )^{-1}  \, {\Vert v \Vert_{1, k, \Omega_\ell}  
\Vert w \Vert_{1, k, \Omega_\ell} } .$$  \\
(ii) 
For all  $v_h, w_h \in \tVhl$, 
\begin{align*} & \max\Big\{{\vert a_{\abs,\ell} (v_h, (\rI - \Pi^h) (\chi_\ell w_h))\vert ,  \vert a_{\abs,\ell} ((\rI - \Pi^h) (\chi_\ell v_h), w_h)\vert}\Big\} \\
& \mbox{\hspace{1.5in}} \ \lesssim\ 
 \left(1 + \frac{1}{k H_\ell } \right)  
\frac {h_\ell}{ \delta_\ell}  \,  {\Vert v_h \Vert_{1,k, \Omega_\ell} \Vert w_h \Vert_{1,k, \Omega_\ell}   } . 
\end{align*}
\end{lemma} 
\begin{proof}
Applying  the Cauchy-Schwarz inequality to \eqref{eq:defb} and using    and \eqref{eq:derivpou}, we obtain
\begin{align*}
\vert b_\ell(v,w) \vert & \ \lesssim  \ {(k \delta_\ell)^{-1}  } \Big(k \Vert w \Vert_{L^2(\tO_\ell)} 
\vert v \vert_{H^1(\tO_\ell)} + k \Vert v \Vert_{L^2(\tO_\ell)} 
\vert w \vert_{H^1(\tO_\ell)}\Big),
\end{align*} 
and the result (i) follows after an application of the Cauchy-Schwarz inequality with respect to the Euclidean inner product 
in $\mathbb{R}^2$.

For (ii), recall  Assumption \ref{ass:basic},  
   use the continuity of $a_{\abs,\ell}$ (from  Lemma 
\ref{lem:cont_coer}) and the fact that  $\Omega_\ell$ has characteristic length scale $H_\ell$  to obtain 
\begin{align} \vert a_\abs(v_h, (\rI - \Pi^h) (\chi_\ell w_h))\vert \ \lesssim \ 
(1 + (kH_\ell)^{-1} ) \Vert v_h \Vert_{1,k,\Omega_\ell} 
\, \Vert (\rI - \Pi^h) (\chi_\ell w_h) \Vert_{1,k,\Omega_\ell}  \ ;
\label{eq:pert1}
\end{align}
the result then follows on  applying Lemma \ref{lem:motivated}. 
\end{proof}
Combining \eqref{eq:comm} with Lemma \ref{lem:estb} and  Theorem \ref{thm:estimate},
we obtain  the following estimate for the quantity $Q_{h,\ell} v_h - \Pi^h (\chi_\ell v_h)$.
As we will see in \eqref{eq:normBinvA}, this quantity is related to the quality of the preconditioner
on the subdomain $\Omega_\ell$. 
 \begin{lemma}\label{cor:rate}
Under Assumption \ref{ass:basic2} and the assumptions of Theorem \ref{thm:estimate}, for all  $v_h \in \cV^h_\ell$, and for all $\ell$, 
\beq\label{eq:Ekey1}
{\Vert Q^h_{\abs,\ell} v_h  - \Pi^h (\chi_\ell v_h) \Vert_{1,k , \tO_\ell}} \ \lesssim \  
\frac{\Cpou} { k \delta_\ell  } \, \Theta(\abs, H_\ell,k) \,
{ \Vert v_h \Vert_{1, k , \tO_\ell} },
\eeq
and
\beq\label{eq:Ekey2}
\Vert Q_{\abs, \ell}^h v_h \Vert_{1,k, \Omega_\ell} \ \lesssim \  
\left[ 1  + \frac {\Cpou}{k \delta_\ell }  \Theta(\eps,H_\ell ,k)  \right] \,  \Vert v_h \Vert_{1, k , \Omega_\ell}.
\eeq
\end{lemma}
  
\begin{proof}
Let $v_h \in \cV_\ell^h$. By \eqref{eq:comm} and \eqref{eq:comm_arrange}, we have 
\begin{align}\label{eq:specialFE}a_{\abs, \ell}(Q^h_{\abs,\ell} v_h - \Pi^h \chi_\ell v_h,  w_{h, \ell}) \ & = F(w_{h,\ell}) , \quad w_{h, \ell} \in \cV^h_\ell , \end{align} 
\beqs \text{where} \quad  F(w_{h,\ell}) \ := \ 
a_{\abs,\ell} ((I-  \Pi^h) (\chi_\ell v_h), w_{h,\ell})  -
a_{\abs,\ell} (v_h, (I - \Pi^h)(\chi_\ell w_{h,\ell}))  +   b_\ell(v_h,w_{h,\ell}).
\eeqs
Using Lemma \ref{lem:estb}  and  \eqref{eq:basic_ass},  we have, for any $w_{h,\ell} \in \cV^h_\ell$,   
\begin{align} 
  \vert F(w_{h,\ell})\vert \ &\lesssim  \  \left( \left(1 + \frac{1}{kH_\ell}\right)  \frac{h_\ell}{\delta_\ell } +
                               \frac{\Cpou}{k \delta_\ell} \right)\label{eq:trouble} 
\Vert  v_h\Vert_{1,k,\Omega_\ell}  \Vert  w_{h,\ell}\Vert_{1,k,\Omega_\ell} \\ 
& = \ \frac{1}{k \delta_\ell} \left( k h_\ell +  \frac{h_\ell}{H_\ell} +1  \right) \Vert  v_h\Vert_{1,k,\Omega_\ell}  \Vert  w_{h,\ell}\Vert_{1,k,\Omega_\ell} \ 
\lesssim  \ \frac{\Cpou}{k \delta_\ell}   \, 
   \Vert  v_h\Vert_{1,k,\Omega_\ell}  \Vert  w_{h,\ell}\Vert_{1,k,\Omega_\ell} , 
\nonumber 
\end{align}
 where we have used   \eqref{eq:basic_ass} and the fact that  $h_\ell \leq H_\ell$.
Then  \eqref{eq:Ekey1}  follows from Theorem \ref{thm:estimate}.
To obtain \eqref{eq:Ekey2}, we write 
$\Vert Q_{\abs,\ell}v_h \Vert_{1,k, \Omega_{\ell}} \  \leq \  \Vert Q_{\abs,\ell}v_h  - \Pi^h(\chi_\ell v_h) 
\Vert_{1,k, \Omega_{\ell}}  + \Vert \Pi^h(\chi_\ell v_h)  \Vert_{1,k, \Omega_{\ell}}$  ,  
and then use   \eqref{eq:teq} and \eqref{eq:Ekey1}.
\end{proof} 
Combining Lemma \ref{cor:rate} with the definition of $\Theta$ in \eqref{eq:Theta1}/\eqref{eq:Theta2},  we have the immediate  corollary:

\begin{corollary}  \label{cor:scenario} 
Under the assumptions of Theorem \ref{thm:estimate},  \\
\noi (i) If $|\abs|>0$, then 
\beq\label{eq:Friday1}
{\Vert Q^h_{\abs,\ell} v_h  - \chi_\ell v_h \Vert_{1,k , \tO_\ell}} \ \lesssim \  
 \frac{k}{|\abs|\delta_\ell} 
   { \Vert v_h \Vert_{1, k , \tO_\ell} }  \,   ;
\eeq
(ii) If $|\abs|\geq 0$, $h\leq \hbarkr$ and each $\Omega_\ell$ is starshaped with respect to a ball uniformly in $\ell$, then
\beq\label{eq:Friday2}
{\Vert Q^h_{\abs,\ell} v_h  - \chi_\ell v_h \Vert_{1,k , \tO_\ell}} \ \lesssim \  
 \left(\frac{H_\ell }{\delta_\ell } + \frac{1}{k \delta_\ell } \right)  
  { \Vert v_h \Vert_{1, k , \tO_\ell} }  \, , \ \text{uniformly  in $\ell$}   .
\eeq

\end{corollary}

\subsection{Bounds on the norm and field of values}\label{sec:norm_fov}

To aid the reader, we recap all the assumptions made so far: 
  Both the fine mesh $\cT_h$ and the subdomains $\{\Omega_\ell\}$ are  assumed shape-regular and have overlap described in \eqref{eq:unifoverlap} and \eqref{eq:op}, with $\delta >0$. We make the finite overlap assumption \eqref{eq:finoverlap}  and the partition of unity functions $\{\chi_\ell\}$ are assumed to be continuous, piecewise linear, and to  satisfy \eqref{eq:derivpou}.  We assume that $k$ and $\eps$ satisfy Assumption \ref{ass:basic},  and either $\eta={\rm sign}(\abs)k$ or $\eta= \sqrt{k^2+ \ri \eps}$.
 All these will be assumed without comment in what follows, but 
  we will explicitly state when we need Assumption \ref{ass:basic2} and the following  slightly stronger assumption.

\begin{assumption} \label{ass:basic3} 
\begin{align} 
k\delta  \rightarrow \infty \quad \text{as} \quad k \rightarrow \infty . \label{eq:kdelta}  
\end{align}
\end{assumption}

This assumption requires   the overlap to contain an increasing number of
  oscillations as $k$ increases (although the rate of increase can be arbitrarily slow).

\begin{theorem}\label{thm:main1}
  Let Assumption  \ref{ass:basic2} hold and suppose that for each $\ell = 1, \ldots , N$ there exists $\sigma_\ell>0$
  such that
\beq\label{eq:sigma}
\N{Q_{\abs,\ell}^h v_h - \Pi^h(\chi_\ell v_h)}_{1,k,\Omega_\ell} \leq \sigma_\ell \N{v_h}_{1,k,\Omega_\ell}, \quad
\text{for all} \quad  v_h\in \cV^h, \quad \ell = 1, \ldots, N. 
\eeq
Set $\sigma = \max\{ \sigma_\ell: \ell = 1, \ldots, N\}$.

\noi (i) Then, 
\beq\label{eq:upperbound}
\max_{v_h \in \cV^h} \frac
 {\N{Q_{\abs}^h v_h}_{1,k}}{{\N{v_h}_{1,k}}} \ 
\lesssim \ 
\Lambda \, \left( 1 + \sigma \right) \ .
\eeq
\noi (ii) If, in addition, Assumption  \ref{ass:basic3} holds, then, for $k$ sufficiently large,
\beq\label{eq:lowerbound}
\min_{v_h \in \cV^h}\, \frac{
\big|(v_h,Q_\abs^h v_h)_{1,k}\big|
}{
\N{v_h}^2_{1,k}
}
\ \geq\ 
\left(
\frac{1}{\Lambda} -  \sqrt{2} \sigma \Lambda\right)  \ + \  R \,   
\eeq
where the remainder $R$
  satisfies  the estimate
\beq \label{HOT}
\left\vert R \right \vert \ \leq  C \,   \frac{\Lambda}{k \delta} \left( 1 + \sigma \right),
\eeq
where $C$ is a constant independent of all parameters. 
Note that  \eqref{eq:lowerbound}  is a genuine lower bound, and the unspecified constant $C$ appears  only in $R$.   
\end{theorem}
\bpf
Throughout the proof, we use the notation
\beq\label{eq:Esplit}
z_l \, :=\,  Q_{\abs,\ell}^h v_h - \Pi^h(\chi_\ell v_h), \quad \text{so that, by \eqref{eq:sigma},} \quad
\Vert z_\ell \Vert_{1,k,\Omega_\ell} \leq \sigma_\ell \Vert v_h \Vert_{1,k, \Omega_\ell} .
\eeq 
To obtain \eqref{eq:upperbound}, we  use  the triangle inequality, then
\eqref{eq:teq} and \eqref{eq:sigma},   to obtain 
\beq\label{eq:QE1}
\Vert Q_{\eps,\ell}^h v_h \Vert_{1,k,\Omega_\ell} \ \leq \ \Vert \Pi^h (\chi_\ell v_h) \Vert_{1,k,\Omega_\ell} \ + \ \Vert z_l  \Vert_{1,k,\Omega_\ell} \ \leq \ \left(1 + \sigma_\ell   \right) \Vert  v_h \Vert_{1,k,\Omega_\ell} . 
\eeq
Then, using Lemma \ref{lem:norm_of_sum},   \eqref{eq:teq} and \eqref{eq:QE1},
\begin{align*}
 \Vert Q_{\eps}^h v_h \Vert_{1,k}^2  &  \ = \
\left\Vert \sum_\ell \Pi^h\left(\chi_\ell Q_{\eps, \ell}^h v_h\right) \right \Vert_{1,k}^2  \ \leq \  \Lambda \sum_\ell \left\Vert  \Pi^h\left(\chi_\ell Q_{\eps, \ell}^h v_h\right) \right \Vert_{1,k, \Omega_\ell}^2 \\
                                     & \lesssim \ \Lambda   \sum_\ell \left\Vert   Q_{\eps, \ell}^h v_h  \right \Vert_{1,k, \Omega_\ell}^2  \
                                       \lesssim  \ 
\Lambda\left( 1 + \sigma  \right)^2  \sum_\ell \left\Vert   v_h  \right \Vert_{1,k, \Omega_\ell}^2 
\end{align*} 
and \eqref{eq:upperbound}  then follows on using \eqref{eq:finoverEuan1}.

To obtain  \eqref{eq:lowerbound}, we
first use  Lemma \ref{lem:motivated},  and \eqref{eq:basic_ass}      to obtain
\begin{align}(v_h, \Pi^h(\chi_\ell Q_{\eps,\ell}^h v_h) ) _{1,k,\Omega_\ell} \ & = \ (v_h, \chi_\ell Q_{\eps,\ell}^h v_h ) _{1,k,\Omega_\ell} + 
                                                                                 \cO \left(\frac{h}{\delta}\right)  \Vert v_h \Vert_{1,k,\Omega_\ell}\, \Vert Q_{\abs, \ell}^h v_h \Vert_{1,k,\Omega_\ell}, \label{last1}\end{align}
Also,  using \eqref{eq:defb} and Lemma \ref{lem:estb}, we have that
\begin{align} 
\big\vert (v_h, \chi_\ell Q_{\eps,\ell}^h v_h ) _{1,k,\Omega_\ell} \ -  \ \ (\chi_\ell v_h,  Q_{\eps,\ell}^h v_h ) _{1,k,\Omega_\ell}\big\vert    & = \ \big\vert b_\ell (v_h, Q^h_{\abs,\ell}v_h)\big\vert \nonumber \\
  & \lesssim   \ 
\cO\left(\frac{\Cpou}{k \delta }\right) \Vert v_h \Vert_{1,k,\Omega_\ell} \, \Vert Q_{\abs, \ell}^h v_h \Vert_{1,k,\Omega_\ell} .  \label{last2}
     \end{align}     
Moreover, by the definition of $z_l$ and Lemma \ref{lem:motivated},
\begin{align}\nonumber  
(\chi_\ell v_h,  Q_{\eps,\ell}^h v_h ) _{1,k,\Omega_\ell} &  \ = \    \ 
\Vert\chi_\ell v_h\Vert^2_{1,k,\Omega_\ell} +(\chi_\ell v_h , z_\ell)_{1,k,\Omega_\ell}  + \big(\chi_\ell v_h, \Pi^h(\chi_\ell v_h) -\chi_\ell v_h\big)_{1,k,\Omega_\ell}\\
&\ =\ \Vert\chi_\ell v_h\Vert^2_{1,k,\Omega_\ell} +(\chi_\ell v_h , z_\ell)_{1,k,\Omega_\ell}  +\cO\left(\frac{h}{\delta}\right)\N{\chi_\ell v_h}_{1,k,\Omega_\ell}\N{ v_h}_{1,k,\Omega_\ell}.
  \label{last3} 
  \end{align} 
Combining \eqref{last1}, \eqref{last2}, and then using \eqref{last3},      we obtain 
\begin{align}
  (v_h, Q_\abs^h v_h)_{1,k} & = \sum_\ell
 \left(v_h, \Pi^h(\chi_\ell Q_{\eps,\ell}^h v_h) \right)_{1,k,\Omega_\ell}
 \nonumber  \\
& \hspace{-1cm}= \sum_\ell \left[ (\chi_\ell v_h , Q_{\eps, \ell}^h v_h)_{1,k,\Omega_\ell} 
                 + \cO\left(\frac{\Cpou}{k \delta} + \frac{h}{\delta}
                 \right)  \Vert v_h \Vert_{1,k,\Omega_\ell}\, \Vert Q_{\abs, \ell}^h v_h \Vert_{1,k,\Omega_\ell}
\right]
\nonumber \\
& \hspace{-1cm} = 
\sum_\ell  
\Big[ \Vert\chi_\ell v_h\Vert^2_{1,k,\Omega_\ell} + 
(\chi_\ell v_h , z_l)_{1,k,\Omega_\ell} \Big]\nonumber   \\
& \hspace{-1cm} \quad + \sum_\ell \left[
\cO\left(\frac{\Cpou}{k \delta} + \frac{h}{\delta} \right)   
\Vert v_h \Vert_{1,k, \Omega_\ell} \, \Vert Q_{\abs,\ell}^h v_h \Vert_{1,k,\Omega_\ell}+ \ \cO\left(\frac{h}{\delta}\right) \N{v_h}_{1,k,\Omega_\ell}  \N{\chi_\ell v_h}_{1,k,\Omega_\ell}
\right]
\label{last4} 
\end{align} 
Using  \eqref{eq:basic_ass},  \eqref{eq:QE1}, \eqref{eq:tech1} and \eqref{eq:finoverEuan1},  the second sum in \eqref{last4} can be estimated by
\begin{align*}
   \frac{1}{k \delta} \sum_\ell (1+ \sigma_\ell) \Vert v_h\Vert_{1,k,\Omega_\ell}^2 \ \lesssim\  \frac{\Lambda (1+\sigma) }{k \delta} \Vert v_h\Vert_{1,k}^2.
\end{align*}  

Also, using the Cauchy-Schwarz inequality, and then  \eqref{eq:tech1} and \eqref{eq:Esplit}, the modulus of the  first sum in \eqref{last4} can be estimated from below by  
\begin{align}
&   \sum_\ell   \Vert \chi_\ell v_h\Vert^2_{1,k,\Omega_\ell} -
    \left| \sum_\ell (\chi_\ell v_h , z_\ell)_{1,k,\Omega_\ell} \right| \ \geq \
    \sum_\ell \left(  \Vert \chi_\ell v_h\Vert^2_{1,k,\Omega_\ell} -
     \Vert \chi_\ell v_h \Vert_{1,k,\Omega_\ell} \Vert z_\ell\Vert _{1,k,\Omega_\ell}\right)  \\
  &\geq \  \sum_\ell   \Vert\chi_\ell v_h\Vert^2_{1,k,\Omega_\ell} - \sqrt{2}  \sigma   
\sum_\ell \Vert v_h \Vert_{1,k,\Omega_\ell}^2 + \mathcal{O}\left(\frac{\sigma}{k \delta} \right)\sum_\ell \Vert v_h\Vert_{1,k,\Omega_\ell}^2 . \label{last5}      
\end{align} 
The result \eqref{eq:lowerbound} then follows from using \eqref{eq:tech4} and \eqref{eq:finoverEuan1}. 
\epf
Using Theorem \eqref{thm:matrixprec}, we now convert this to a statement about matrices, 
\begin{theorem} \label{thm:main_intro} 
Let Assumption \ref{ass:basic2} hold and let  $\sigma>0$ be such that
\begin{align}  \label{localstar}
\Vert A_{\abs,\ell}^{-1} R_\ell A_\abs - R_\ell \Vert_{D_k} \ \leq \ \sigma    \ , \quad \ell = 1, \ldots, N .  
\end{align} 
Then 
\begin{align} \label{eq:normestintro} 
\Vert B_\abs^{-1} A_\abs \Vert_{D_k} \ \lesssim  \  {   \Lambda \left( 1 + {\sigma}\right)}\ .   
\end{align}
If, in addition, Assumption \ref{ass:basic3} holds then, for $k$ sufficiently large,
\begin{align}\label{eq:fovintro} 
 \min_{\bV \in \mathbb{C}^n} \, \frac{ \left\vert \langle \bV, B_\abs^{-1} A_\abs \bV\rangle_{D_k}\right\vert }
  {\Vert \bV \Vert_{D_k}^2    } \ \geq \ \left(\frac{1}{\Lambda} - \sqrt{2} \sigma \Lambda\right) + R. 
\end{align}
with $R$ satisfying \eqref{HOT}.  
\end{theorem}

\begin{proof}
First  note  that, from \eqref{eq:wip} and \eqref{eq:HelmE}, if $v_h \in \cV^h$ is a finite element function with nodal vector $\bV$, then
  $ \Vert v_h \Vert_{1,k} = \Vert \bV \Vert_{D_k}$. By Theorem \ref{thm:matrixprec}, the nodal vectors of   $Q_{\abs,\ell}^h v_h$ and $Q_\abs^h v_h$ 
   are $A_{\abs,\ell}^{-1} R_\ell A_{{\abs}} \bV$  and $B_\abs^{-1} A_\abs \bV$  respectively.  By \eqref{eq:explicit},  the nodal vector of $\Pi^h(\chi_\ell v_h) $ is $R_\ell \bV$.  Thus
\beq\label{eq:normBinvA}
 \Vert Q_\abs^h \Vert_{1,k} = \Vert B_\abs^{-1} A_\abs\Vert_{D_k} \quad \text{and} \quad \Vert Q_{\eps,\ell}^h v_h - \Pi^h(\chi_\ell v_h) \Vert_{1,k} = \Vert A_{\abs,\ell}^{-1} R_\ell A_{{\abs}} \bV - R_\ell \bV \Vert_{D_k} .
 \eeq
From these relations, and also \eqref{eq:nodal_add},  Theorem \ref{thm:main1} implies  Theorem \ref{thm:main_intro}. 
\end{proof}

This theorem immediately yields the following corollary about the convergence of GMRES.

\begin{corollary}\label{cor1_intro}
  Suppose Assumptions \ref{ass:basic2} and  \ref{ass:basic3},  hold and that \eqref{localstar} holds,  with  
  \begin{align}
    \label{eq:sigsmall}
    \sigma \ < \ \frac{1}{\sqrt{2} \Lambda^2}. 
    \end{align}
If  GMRES is applied to    
\eqref{eq:discrete} in the inner product induced by $D_k$ with $B_\eps^{-1}$ as a left   preconditioner, 
then the number of iterations needed  to achieve a   {prescribed} accuracy remains bounded as
$k \rightarrow \infty$.    
\end{corollary}

\bpf
This follows directly from Theorem \ref{thm:main_intro} and the GMRES convergence theory in \cite{EiElSc:83}.
\epf
As explained above, Assumptions  \ref{ass:basic2} and \ref{ass:basic3} are quite mild requirements.  
However \eqref{eq:sigsmall} is a stronger constraint and {may lead to}  restrictions  on  $\abs$ and $H$.
Essentially it  says 
that for each $\ell$, the   ``local impedance solve''  $A_{\abs, \ell}^{-1} $  should be 
a sufficiently good left inverse for   $A_\abs$ when it is restricted to  $\overline{\Omega_\ell}$.
In the following corollary, whose proof follows  from Corollary \ref{cor:scenario}, Part (i)   gives conditions under which $\sigma$ can be bounded (hence useful for the upper bound \eqref{eq:normestintro}), while Parts (ii) and (iii) give conditions for  $\sigma$ to be  small
(and hence are  relevant to ensuring \eqref{eq:fovintro}).}

\begin{corollary}  \label{cor:sigma}  
Let the assumptions of Theorem \ref{thm:main1} hold.

{\noi (i) Assume that $h\leq \hbarkr$, and each $\Omega_\ell$ is starshaped with respect to a ball uniformly in $\ell$.
 Then, for all $\abs$ with ${0\leq}\vert \abs \vert \leq k^2$,    we have $\sigma \ \lesssim \ H/\delta$.}

\noi (ii) If $|\abs|>0$, $\abs \sim  k^{1+\beta} $ for $0<\beta<1$, $\delta\sim H \sim  k^{-\alpha}$ for $0<\alpha<1$, then 
$
\sigma\ \lesssim\   k^{ \alpha-\beta}.
$

\noi (iii) If $|\abs| >  0$  and $\delta$ is fixed, then there exist  constants  $C$ and $k_0$ so that when $\abs = Ck$ and $k\geq k_0$,
$$
\sigma \  {\leq} \ \frac{1}{{2}\sqrt{2} \Lambda^2 } . 
$$
\end{corollary} 
Using the bounds of Corollary \ref{cor:sigma} in Theorem \ref{thm:main_intro}, we obtain the
following results about $B_\abs^{-1} A_\abs$.

\begin{corollary}[Upper bound on  the norm of $B_\abs^{-1}A_\abs$]\label{cor:upper_bound}
Assume that 
$h\leq \hbarkr$, and each $\Omega_\ell$ is starshaped with respect to a ball uniformly in $\ell$.
Assume that $\delta\sim H$.  Then, {for all $0\leq |\abs|\leq k^2$}, 
\beqs
\Vert B_\abs^{-1} A_\abs \Vert_{D_k} \ \lesssim \ 1 . 
\eeqs
\end{corollary}

\begin{corollary}[Lower bound on the distance of the field of values from the origin]\label{cor:lower_bound}

\noi (i) If $|\abs| \sim  k^{1+\beta} $ for $0<\beta<1$, $\delta\sim H$, and $H \sim  k^{-\alpha}$ for $0<\alpha<1$ then 
\beqs
\quad  \min_{\bV \in \mathbb{C}^n}
\frac{
  \big|\langle \bV, B_\abs^{-1} A_\abs \rangle_{D_k}
  \big|
}{
\Vert \bV \Vert ^2_{D_k} 
}\ \geq \  1 - \cO(k^{\alpha-\beta}) ,  \quad \text{as} \quad k\tendi.
\eeqs

\noi (ii) If $\delta$ is fixed, then there exist constants  $C$ and $k_0$ so that when ${|}\abs{|} = Ck$ and $k\geq k_0$,
  \beqs
\min_{\bV \in \mathbb{C}^n}
\frac{
  \big|\langle \bV, B_\abs^{-1} A_\abs \rangle_{D_k}
  \big|
}{
\Vert \bV \Vert ^2_{D_k} 
}\ \geq \  \frac{1}{2\Lambda} \ . 
\eeqs
\end{corollary}

\bre[Right preconditioning]\label{rem:leftright}{The results in \cite[Theorem 5.8]{GrSpVa:17} -- see also  \cite[{\S3}]{GrSpVa:17a}} --  show how results about right
preconditioning {(working in the $D_k^{-1}$ inner product)} can be obtained from analogous results about left preconditioning of the adjoint problem {(working in the $D_k$ inner product)}. 
The results in \S\S\ref{sec:fem_and_ddm}, \ref{sec:3.1}, and \ref{sec:3.2} all hold when the problem \eqref{eq:PDE}, \eqref{eq:ImpBC} is replaced by its adjoint (see, in particular, Remark \ref{rem:adj}); therefore the results in this section about left preconditioning {(in the $D_k$ inner product)} also hold
for right preconditioning {(in the $D_k^{-1}$ inner product)}.
\ere

\bre[Dirichlet boundary conditions]\label{rem:localD}
Some parts of the analysis presented in this paper hold in the case when the boundary conditions on the subdomains are changed from impedance to Dirichlet, i.e.~when the integral over $\partial \Omega_\ell\setminus \Gamma_D$ is removed from \eqref{eq:localimp}. However, Parts (ii) of Theorem \ref{thm:estimate} and Corollary \ref{cor:scenario} no longer hold in this case. {{Additionally,}  the upper bound on the norm for  $\eps=0$ in Corollary \ref{cor:upper_bound} does not hold either. We see in Experiment \ref{expt6} that when the impedance boundary conditions are replaced by Dirichlet, the preconditioner performs poorly for the pure Helmholtz equation. }
\ere

\section{Numerical Experiments} 
\label{sec:Numerical}

In this section we give numerical  experiments illustrating  the 
performance of   the preconditioners defined in \S \ref{subsec:prec} and analysed in \S  \ref{sec:norm_fov}. 
We consider  Problem \eqref{eq:PDE}-\eqref{eq:ImpBC} with $\Omega$ being the 
unit square in 2-d. We  first choose a uniform coarse mesh $\cT^H$ of equal square elements of 
side length  $H = 1/M$ on $\Omega$.   
Let $\bx_{\ell,m} =  (\ell H , m H), \  \ell, m = 0, \ldots , M$ denote the  coarse  mesh nodes. 
We introduce subdomains $ \Omega_{\ell,m} $,
defined to be interior of  the union of all the coarse mesh elements that touch $\bx_{\ell,m}$, for  $\ell, m = 0, \ldots , M$.    
These subdomains have generous overlap in the sense of  \eqref{eq:unifoverlap}. 
Let 
$\chi_{\ell,m}$ denote the piecewise bilinear nodal basis functions with respect to the coarse mesh, 
i.e. $\chi_{\ell,m}$ is bilinear with respect to the coarse mesh and  $\chi_{\ell,m} (\bx_{\ell', m'}) = \delta_{{\ell- \ell'},{m - m'}}$.   Then $\{ \chi_{\ell,m} : \ell, m = 0, \ldots ,M\}$ form a 
partition of unity and we use this to define the preconditioner \eqref{eq:ASpc}.   

\noindent {\bf   Illustration of the notation used in the proof of \eqref{eq:tech3}}:   \  
Recalling \eqref{POUstar},  we can see that for each $\ell,m \in \{ 0, \ldots, M\}$,
$\supp \chi_{\ell,m}  \subseteq \overline{\Omega_{\ell,m}}$. Moreover, for $\bx \in \oOmega$ and $m(\bx)$ defined by \eqref{eq:defm}, we have 
\begin{align*}
  m(\bx ) = 1 \quad & \text{when} \  \bx \ \text{is a node} \quad \bx \in \{ \bx_{\ell,m} : \ \ell, m \in \{ 0,\ldots,  M\} \}\\
  m(\bx) = 2 \quad & \text{when} \ \bx \ \text{is an interior point of any edge of the coarse mesh}\\
  m(\bx) = 4 \quad & \text{when} \ \bx \ \text{is an interior point of any coarse mesh element} 
\end{align*}
Hence  $\Lambda = 4$. Note that  $D_1$ contains all the nodes of the coarse grid, $D_2$ contains all interior points of edges of the coarse grid and $D_4$ contains all interior points of coarse grid elements. 
Note  $\mu(D_1) = \mu(D_2) = 0$, with $\mu$ denoting Lebesgue measure,  and  $\mu(D_4) = \mu(\Omega). $
Moreover the index sets $\cD(1), \cD(2) $ and $\cD(4)$ actually contain all indices $(\ell,m)$ with $\ell, m \in \{ 0, \ldots , M\}$.


The coarse mesh is then refined uniformly to obtain a fine triangular mesh 
$\cT^h$. The space $\cV^h$ which is used to obtain the linear system \eqref{eq:discrete} is  the  
space of piecewise-linear finite-element functions on $\cT^h$.  The linear system \eqref{eq:discrete} is therefore 
characterised by two parameters: 
the fine mesh diameter $h$  and $\eps$ in \eqref{eq:PDE} denoted by 
$ \hprob \ \text{and} \  \epsprob$
respectively.
In all the experiments here we choose $h \sim k^{-3/2}$ (the level of 
refinement generally believed to keep the relative error of the finite-element solution bounded independently of $k$ as $k\tendi$; see Remark \ref{rem:accuracy}). Although these are 2d problems, the dimension $n = (k^{3/2})^2 = k^3$ of the systems grows very quickly with $k$ , and is  well over $10^6$ when $k = 140$ (considered below).
The preconditioner  is characterised by the coarse grid diameter and the  level of absorption used, denoted by
$ H \ \text{and} \  \epsprec$
respectively. 

In Experiments   \ref{expt2} and \ref{expt3}, we verify the theory by illustrating the performance of the preconditioner on some problems with $\epsprob >0$. In Experiments
\ref{expt4}, \ref{expt5}, and \ref{expt6}, we solve the ``pure Helmholtz" problem, i.e.~$\epsprob =0$. 
Unless otherwise stated, the data   
${f}, g $ in \eqref{eq:discrete} is chosen so that the exact solution of \eqref{eq:vp} - \eqref{eq:Helmholtzvf_intro} is a plane wave $u(x) = \exp(\ri k x.\widehat{d})$ where $\widehat{d} = (1/\sqrt{2}, 1/\sqrt{2})^\top$. Note that oscillations in the solution are resolved by the fine grid but are not resolved by the subdomains. 
We choose $\Gamma_D=\emptyset$, so that $\Gamma=\Gamma_I$.
Except in Experiment \ref{expt5},   the initial guess for GMRES 
is chosen to be a random 
(uniformly distributed in $[0,1]^m$) vector in $\mathbb{R}^n$. 
In all cases the GMRES stopping criterion  is based on requiring the initial residual to be reduced by $10^{-6}$. 
Standard GMRES (with  residual minimisation in the Euclidean norm)  is used,  even though the estimates in 
Theorem \ref{thm:main_intro} are with respect to the norm induced by $D_k$; the numerical experiments in \cite{GrSpVa:17} and  in \cite{BoDoGrSpTo:17c}  (for a similar method) found the iteration counts to be essentially identical when minimisation in the Euclidean norm is replaced by minimisation in the norm induced by $D_k$. 


\begin{experiment} \label{expt2} 
We choose
\begin{align} \label{eq:case1} 
\hprob  \sim k^{-3/2}, \quad \epsprob = \epsprec= k^{1 + \beta}, \  \Hprec = k^{-\alpha} \ ,  \quad \text{where} \quad \beta =  \alpha + 0.1.   \end{align} 
\end{experiment} 
Corollary \ref{cor:sigma}  predicts a wavenumber-independent iteration count for GMRES and this  
behaviour is clearly visible in Table \ref{tab:1} (left).   
Reading across this table, for fixed $k$,  larger    $\alpha$ corresponds to smaller subdomains (and thus the   
preconditioner becomes cheaper per iterate). The number of iterations increases (slightly) as $\alpha$ increases but remains bounded as $k$ increases for fixed $\alpha$. We also note that if we read diagonally across Table \ref{tab:1}(a) (thus increasing the rate of decrease of $H$ as $k$ increases)  we see roughly  logarithmic growth in the number of iterations, although the analogous growth is somewhat faster  in later tables.

\begin{table}[ht]
\centering
\subfloat[Subtable 1 list of tables text][GMRES iterations for case  \eqref{eq:case1}]
{
\begin{tabular}{|c|cccc|}
\hline 
$k\backslash \alpha$ &0.2 &0.3 &0.4 &
 0.5\\
\hline 
40 &4  &6 & 7& 9\\
60 & 4 &5 &7&10\\
80 &3 & 6&8 &9\\
100&5 & 6 & 7 & 9\\
120 & 4 & 5 & 7&9\\
140 & 4 &5 & 7&9\\
\hline 
\end{tabular} 
}
\qquad
\subfloat[Subtable 2 list of tables text][GMRES iterations for  case \eqref{eq:caseA}.]{
\begin{tabular}{|c|cccc|}
\hline 
$k\backslash \alpha$ &0.2 &0.3 &0.4 &
 0.5\\
\hline 
40 &4 & 7 & 10& 17\\
60 & 4 & 7 & 12 & 22\\
80 & 4 & 9 & 13 & 21\\
100& 6 & 8 & 13 & 23\\
120 & 5 & 8 & 15 & 24\\
140 & 5 & 7 & 13 & 25 \\
\hline 
\end{tabular} 
}
\caption{\label{tab:1}}
\end{table}


Based on Experiment \ref{expt2}, and recalling the discussion in the introduction (in particular Equation \eqref{eq:c}), we now investigate how well the preconditioner performs when we reduce the absorption in the problem being solved to $\epsprob = k$.

\noindent 
\begin{experiment}  \label{expt3} We choose
\begin{align} \label{eq:caseA} 
\hprob  \sim k^{-3/2}, \quad \epsprob = \epsprec=k, \ \tand \  \Hprec = k^{-\alpha}.    \end{align} 

       
\end{experiment} 
Comparing Tables \ref{tab:1} (left and right), we see an increase in the iteration 
numbers (especially for larger $\alpha$) but growth with $k$ appears to be avoided provided  $\alpha \leq 0.4$.   
This shows that   $B_k^{-1}$ is a good preconditioner for $A_k$ and so by the heuristic argument centred on \eqref{eq:c}, we expect 
$B_k^{-1}$  to be  good preconditioner for $A$.  Experiment \ref{expt4} shows this to be true.
Here $\epsprob$ is reduced from $k$ to $0$; we  see a slight increase in iteration numbers compared to  $\epsprob  = k$,
but still apparent robustness to increasing  $k$,  for fixed    $\alpha \leq 0.4$.     

\begin{experiment} \label{expt4}  We choose
\ \begin{align} \label{eq:case4} 
\hprob  \sim k^{-3/2}, \ \epsprob = 0, \ \tand \ \Hprec = k^{-\alpha}.
\end{align} 


\begin{table}[h] 
\begin{center} 
\begin{tabular}{|c||cccc||cccc|}
\multicolumn{1}{c}{} & \multicolumn{4}{c}{$\epsprec = k$} & \multicolumn{4}{c}{$\epsprec = 0$}\\
\hline 
$k\backslash \alpha$ &  0.2& 0.3 &0.4 & 0.5 & 0.2&  0.3 &0.4 & 0.5\\
\hline 
40 &6  &8 &12& 20 &5  &8 & 11&19  \\
60 & 5  &8 &14 & 25 &5  &7& 14&25 \\
80 & 5  &10  &15  & 25 & 4  &10 &15 & 24  \\
100 & 7  &9 &15 & 27  & 7  &9 &15 & 27  \\
120 & 6 & 9 & 17 & 29 & 6 &9 & 17 &29\\
140 & 6&9&17& 31 & 6 &8 & 16 & 31 \\
\hline 
\end{tabular} 
\vspace{0.2cm} 
\caption{Number of GMRES iterations  for the case \eqref{eq:case4}. 
 \label{tab:4} } 
\end{center} 
\end{table} 

\end{experiment} 

We make two observations  from the results of Experiments \ref{expt2}-\ref{expt4}.
\ben
\item
The one-level Schwarz method provides an optimal 
preconditioner for the pure Helmholtz problem; -- the iteration numbers appear  bounded independently 
of $k$ (and hence $n$) as $ k $ increases --  provided the subdomain diameter does not shrink 
too quickly. Robustness is maintained when the subdomain diameters  shrink no faster than   $\mathcal{O}(k^{-0.4})$. 
\item The performance of the  preconditioner is virtually the same 
whether  it is built from the absorptive system  $\epsprec = k$ or from the pure Helmholtz 
system $\epsprec = 0$. Whilst the results of the present paper give theoretical support for the observed robustness when $\epsprec = k$,  (see the discussion in \S\ref{sec:intro} and Appendix \ref{sec:app1}); with existing theoretical tools it seems very difficult to prove results for the case $\epsprec = 0$.  
\een


\begin{experiment} \label{expt5}
As a more extreme case we consider  subdomains which  are
fixed as $h \rightarrow 0$. While this is not a practical method (the subproblems have the same order of complexity as the global problem), 
  it can provide a useful starting point for methods based on recursive
    application of the one-level method, as described in \S \ref{sub:pract}. We therefore consider:
\begin{align} \label{eq:case5} h \sim k^{-3/2}, \ \epsprob = 0,\  H = 1/M. \end{align}  
In the left-hand panel of Table \ref{tab:5}, $\epsprec = k$ and  a random starting guess is chosen.
In the middle panel, $\epsprec = 0$ and a random starting guess is chosen.
In the right-hand panel,  $\epsprec = k$ and a  zero starting guess is chosen.  
Again, there is little effect from 
switching off the absorption in the preconditioner.  Surprisingly  a random starting guess leads to consistently 
lower iteration counts than a zero starting guess; we have no explanation for this  observation.  
\begin{table}[h]
\begin{center} 
\begin{tabular}{|c||ccc||ccc||ccc|}
\hline 
\multicolumn{1}{|c|}{} & \multicolumn{3}{|c|}{random} & \multicolumn{3}{|c|}{random} & \multicolumn{3}{|c|}{ zero  }\\
\multicolumn{1}{|c|}{} & \multicolumn{3}{|c|}{starting guess} & \multicolumn{3}{|c|}{starting guess} & \multicolumn{3}{|c|}{ starting guess }\\
\multicolumn{1}{|c|}{} & \multicolumn{3}{|c|}{$\epsprec = k$} & \multicolumn{3}{|c|}{$\epsprec = 0$} & \multicolumn{3}{|c|}{ $\epsprec = k$}\\
\hline 
$k\backslash M$ &  4 & 8 & 16 &4 & 8 & 16 & 4 & 8 & 16 \\
\hline 
40 & 12 & 27 & 61 & 11 & 27 & 61 & 16 & 36 & 82  \\
60 &11  &  25 & 56 & 10 & 25 & 56 & 15 & 36 & 81\\
80 & 10 & 22 &52   & 10 & 22 & 52 & 15 & 33 & 75 \\
100 & 9 & 21 & 48 & 9 & 21 & 48 & 15 & 33 & 71 \\
120 &9 & 20  &45 & 9 & 20 & 45 & 15 & 31 & 69 \\
140 &8  &18 &41  & 8 & 18 & 41 & 14 & 31 & 70 \\
\hline 
\end{tabular} 
\vspace{0.2cm} 
\end{center} 
\caption{Number of  GMRES iterations for the case \eqref{eq:case5} \label{tab:5}} 
\end{table}

\end{experiment}  

Finally  we study the effect of changing the boundary condition on the subdomains from Impedance to Dirichlet (recall Remark \ref{rem:localD}). 

\begin{experiment} \label{expt6}  We choose Dirichlet conditions on subdomains with 
\ \begin{align} \label{eq:case6} 
\hprob  \sim k^{-3/2}, \ \epsprob = 0, \   \tand \ \Hprec = k^{-\alpha}.
  \end{align}

In Table \ref{tab6} we see that this yields a  
very poor preconditioner  for the pure Helmholtz problem (compare Experiment \ref{expt6} with Experiment \ref{expt4}).
Similar observations are made in \cite{GrSpVa:17}, where coarse grids were also used to improve the robustness.

\begin{table}[h] 
\begin{center} 
\begin{tabular}{|c||cccc||cccc|}
\multicolumn{1}{c}{} & \multicolumn{4}{c}{$\epsprec = k$} & \multicolumn{4}{c}{$\epsprec = 0$}\\
\hline 
$k\backslash \alpha$ &  0.2& 0.3 &0.4 & 0.5 & 0.2&  0.3 &0.4 & 0.5\\
\hline 
10 & 7 & 7 & 12 & 12 & 6 & 6 & 15 & 15\\
20 & 7 & 7 & 17 & 25 & 5 & 5 & 20 & 29\\
40 & 6 & 16 & 34 & 86 & 5 & 22 & 43 & 110 \\
60 & 6 & 16 & 68 & 102 & 5 & 25 & 83 & 121\\
80 & 5 & 46 & 127 & 239 & 5 & 78 & 173 & 256 \\
100 & 14 & 58 & 130 & 242 & 22 & 121 & 222 & 429 \\
\hline 
\end{tabular} 
\vspace{0.2cm} 
\caption{Number of GMRES iterations  for the case  \eqref{eq:case6} with homogeneous Dirichlet condition on subdomain boundaries  
\label{tab6} } 
\end{center} 
\end{table} 
\end{experiment} 
\begin{appendix}
\section{A rigorous basis for the discussion around \eqref{eq:c}}\label{sec:app1}
\begin{lemma} \label{lem:A1} 
  Let  $(\cdot, \cdot)$ be an inner product with associated norm $\Vert \cdot\Vert$.
Assume that \eqref{eq:GGS} holds with $\|\cdot\|_2$ replaced by $\|\cdot\|$ and with $K>0$ independent of $\abs$ and $k$.
Assume also that for all
  $\abs$ in some  neighbourhood of the origin, there exist positive numbers  $C_1(\abs)$  and $C_2(\abs)$
  (which may depend on $\abs$
  but are independent of all other parameters), such that
  \begin{align}
    \Vert B_\abs^{-1} A_\abs \Vert \ \leq C_1(\abs) ,
\quad\tand\quad
    \frac{\vert (\bV, B_\abs^{-1} A_\abs \bV)\vert}{\Vert \bV \Vert^2} \ \geq \ C_2(\abs)      \quad  \text{for all} \quad \bV \in \mathbb{C}^n .
    \label{eq:A2}
    \end{align}
    Then
    \begin{align}\label{eq:A3} \Vert B_\abs^{-1} A \Vert \ \leq \ C_1(\abs) \left( 1+ K \frac{|\abs|}{k}\right) \quad\tand\quad
    \frac{\vert (\bV, B_\abs^{-1}A \bV ) \vert} {\Vert \bV \Vert^2 } \ \geq \ C_2(\abs) - K \, C_1(\abs) \frac{|\abs|}{k}
    \end{align}
for all $\bV \in \mathbb{C}^n$.
\end{lemma}

\bre
Observe that for the norm in \eqref{eq:A3} to remain bounded we simply need $C_1(\abs)$ to be bounded, while for the field of values to be bounded away from the origin we need the stronger condition
  $${C_2(\abs)} \ > \  K C_1(\abs) \frac{|\abs|} {k}.$$
\ere

\begin{proof}[Proof of Lemma \ref{lem:A1}]
  The first bound in \eqref{eq:A3} follows from  \eqref{eq:pert}, \eqref{eq:GGS}, and the first equation in \eqref{eq:A2}.  To obtain the second bound in \eqref{eq:A3}, we use \eqref{eq:pert}, the first bound in \eqref{eq:A3},
  and the inverse triangle inequality to obtain
\beqs
  \big\vert ( \bV , B_\abs^{-1} A \bV ) \big\vert\  \geq \ \big\vert ( \bV , B_\abs^{-1} A_\abs \bV ) \big\vert - K C_1(\abs)  \frac{|\abs|}{k} \Vert \bV\Vert^2 ,
\eeqs
  and we then use the second equation in \eqref{eq:A2}. 
  \end{proof}

\end{appendix}

\bibliographystyle{siam}
\bibliography{biblio_epsilon3}

\end{document}